%% file: EstQ_main.tex
\documentclass[article,11pt]{amsart}

\usepackage[T1]{fontenc}
\usepackage[applemac]{inputenc}
\usepackage[francais,english]{babel}
\usepackage{amssymb}
\usepackage{amsmath}
\usepackage{latexsym}
\usepackage{marvosym}
\usepackage{amsfonts}
\usepackage{graphicx}
\usepackage{hyperref}
\usepackage{color}
\usepackage{cancel}

\usepackage{enumerate}
\hypersetup{
backref=true, 
pagebackref=true,
hyperindex=true, 
colorlinks=true, 
breaklinks=true, 
urlcolor= blue, 
linkcolor= blue, 
bookmarks=true, 
bookmarksopen=true, 
pdfauthor={Romain AZAIS}, 
}

\newcommand{\prob}{\mathbf{P}}
\newcommand{\esp}{\mathbf{E}}
\newcommand{\ud}{\text{\normalfont d}}
\newcommand{\var}{\mathbf{V}\text{\normalfont ar}}
\newcommand{\tops}{\stackrel{a.s.}{\longrightarrow}}
\newcommand{\cov}{\mathbf{C}\text{\normalfont ov}}

\newcommand{\R}{\mathbf{R}}

\renewcommand{\phi}{\varphi}

\newcommand{\fin}{
$\hfill
\mathbin{\vbox{\hrule\hbox{\vrule height1.5ex \kern.6em
\vrule height1.5ex}\hrule}}$}

\usepackage{geometry}
\newtheorem{prop1}{Proposition}[section]
	
\newtheorem{lem1}[prop1]{Lemma}
\newtheorem{cor1}[prop1]{Corollary}
\newtheorem{theo}[prop1]{Theorem}
\newtheorem{hyp}[prop1]{Assumption}
\newtheorem{hyps}[prop1]{Assumptions}
\newtheorem{rem}[prop1]{Remark}

\email{\vspace{-0.2cm}romain.azais@inria.fr}



\begin{document}
\title[A recursive nonparametric estimator for the transition kernel of a PDMP]
{A recursive nonparametric estimator for the transition kernel of a piecewise-deterministic Markov process}
\author{Romain Aza\"{\i}s}

\address{
INRIA Bordeaux Sud-Ouest, team CQFD, France and Universit\'e Bordeaux, IMB, CNRS UMR 5251,
200, Avenue de la Vieille Tour, 33405 Talence cedex, France.
}

\thanks{This work was supported by ARPEGE program of the French National Agency of Research (ANR), project “FAUTOCOES”, number ANR-09-SEGI-004.}

\begin{abstract}
In this paper, we investigate a nonparametric approach to provide a recursive estimator of the transition density of a piecewise-deterministic Markov process, from only one observation of the path within a long time. In this framework, we do not observe a Markov chain with transition kernel of interest. Fortunately, one may write the transition density of interest as the ratio of the invariant distributions of two embedded chains of the process. Our method consists in estimating these invariant measures. We state a result of consistency and a central limit theorem under some general assumptions about the main features of the process. A simulation study illustrates the well asymptotic behavior of our estimator.
\end{abstract}

\maketitle


\input{EstQ_intro}

\input{EstQ_text01}
\input{EstQ_text02}


\nocite{*}
\bibliographystyle{acm}
\bibliography{EstQ_main} 

	\begin{figure}[!p]
	\includegraphics[width=0.7\textwidth]{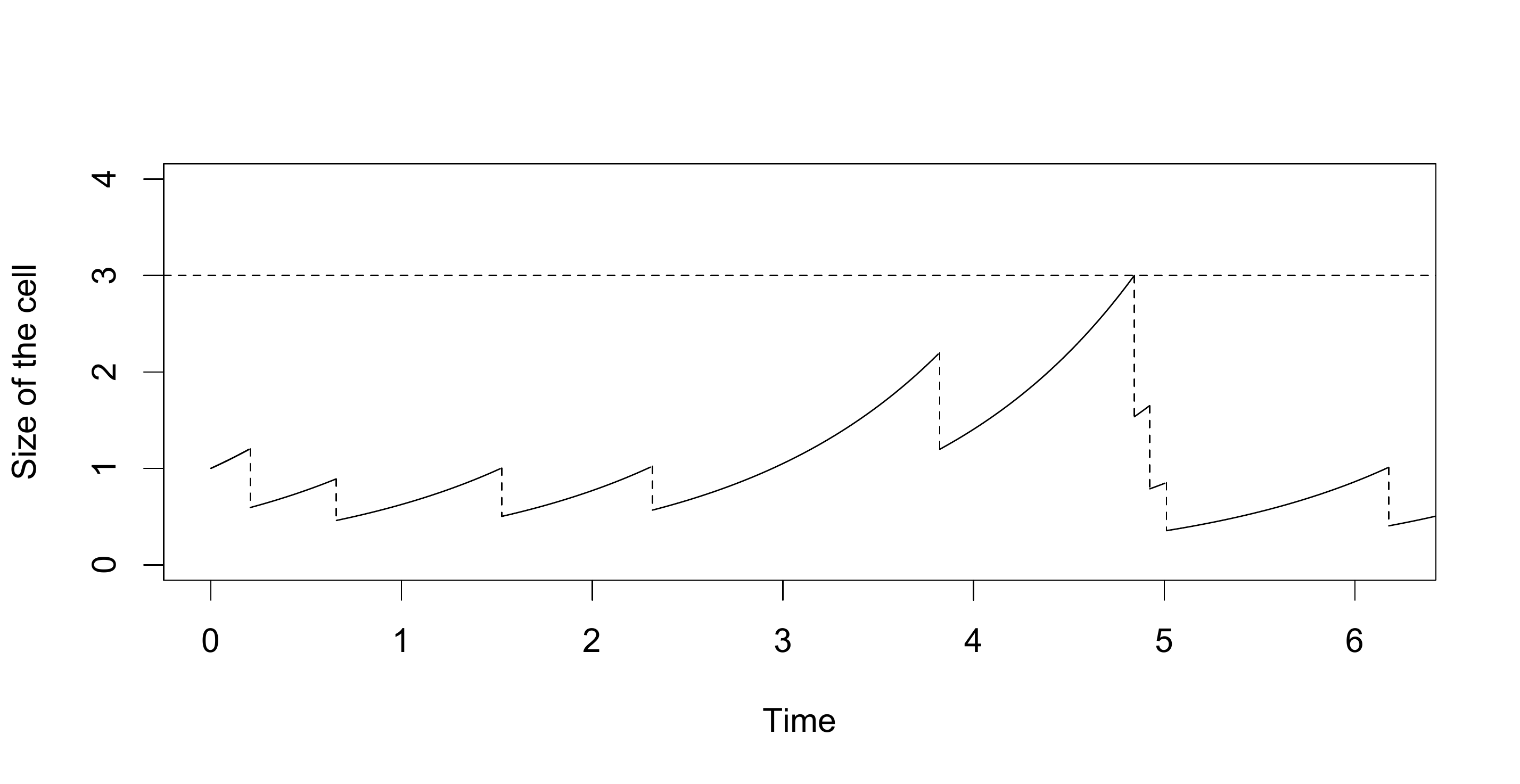}
	\caption{Evolution of the size of the cell when $10$ jumps occur.}
	\label{fig:trajectoire}
	\end{figure}
	
	\begin{figure}[!p]
	\begin{tabular}{cc}
	\includegraphics[width=0.49\textwidth,height=0.28\textwidth]{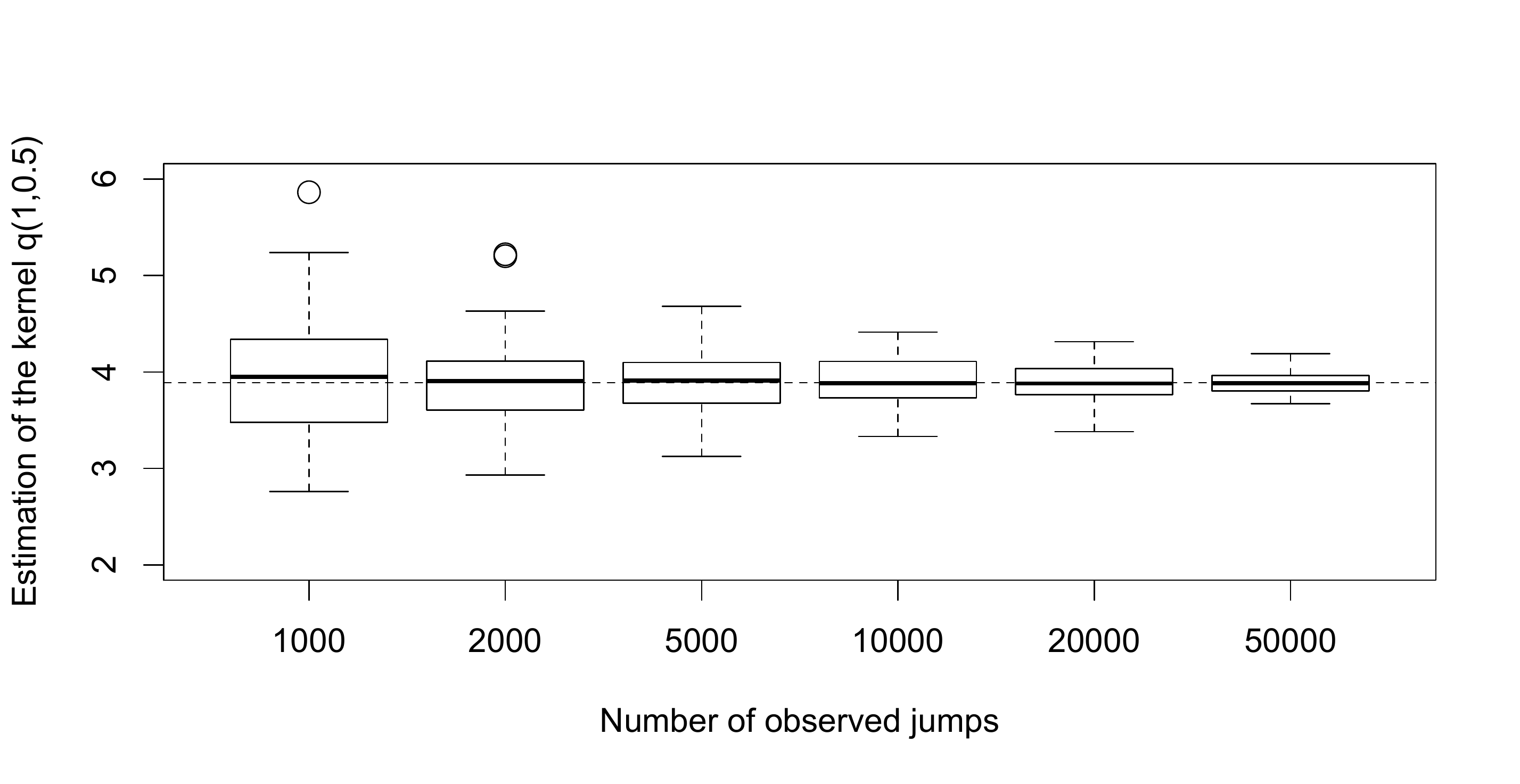} & \includegraphics[width=0.49\textwidth,height=0.28\textwidth]{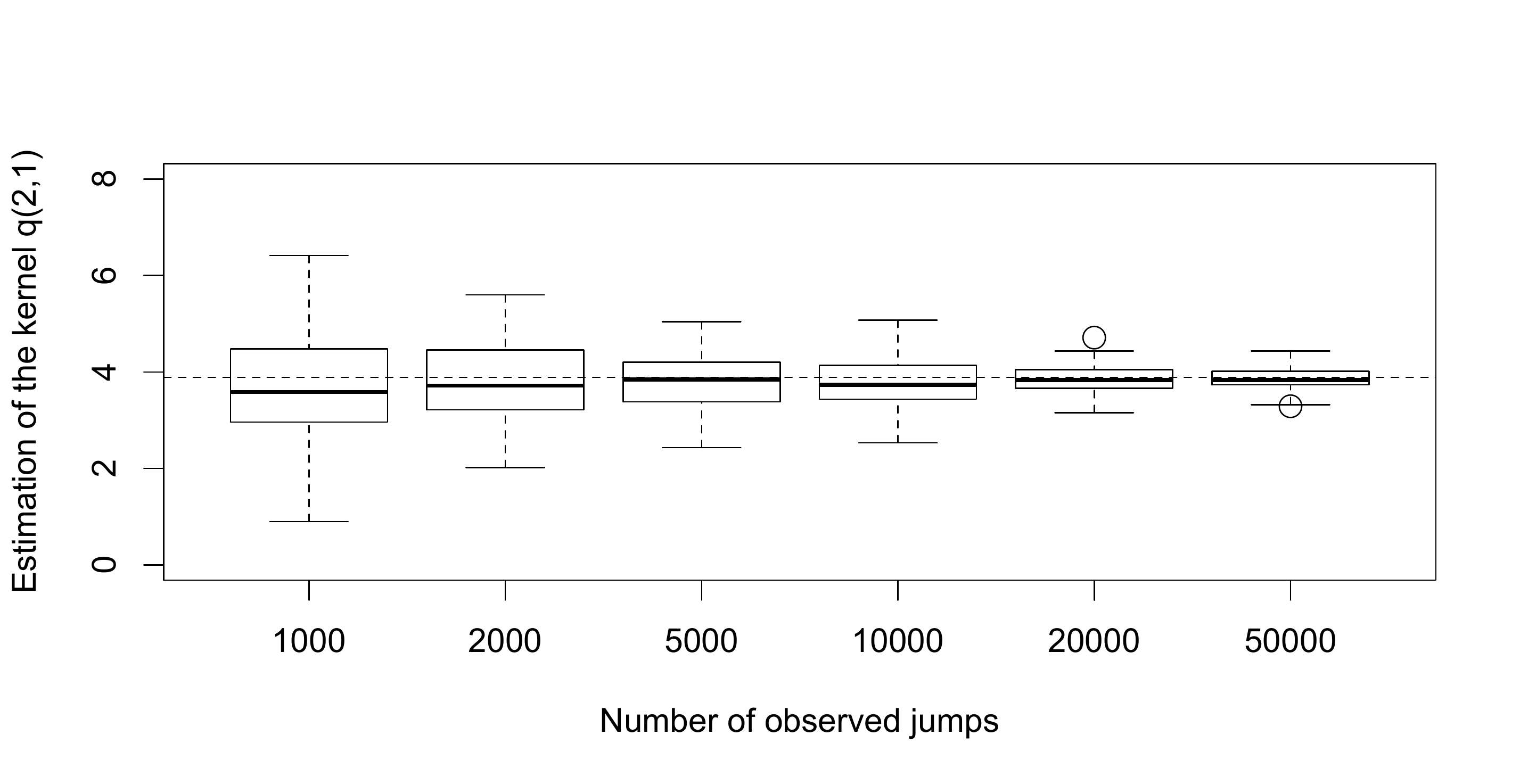}
	\end{tabular}
	\caption{Boxplots over $100$ replicates of the estimator $\widehat{q}_n(x,y)$ of $q(x,y)$ for $x=1$ and $y=0.5$ (left) and $x=2$ and $y=1$ (right) from different numbers of observed jumps.}
	\label{fig:boxplots}
	\end{figure}

	\begin{figure}[!p]
	\begin{tabular}{cc}
	\includegraphics[width=0.49\textwidth,height=0.28\textwidth]{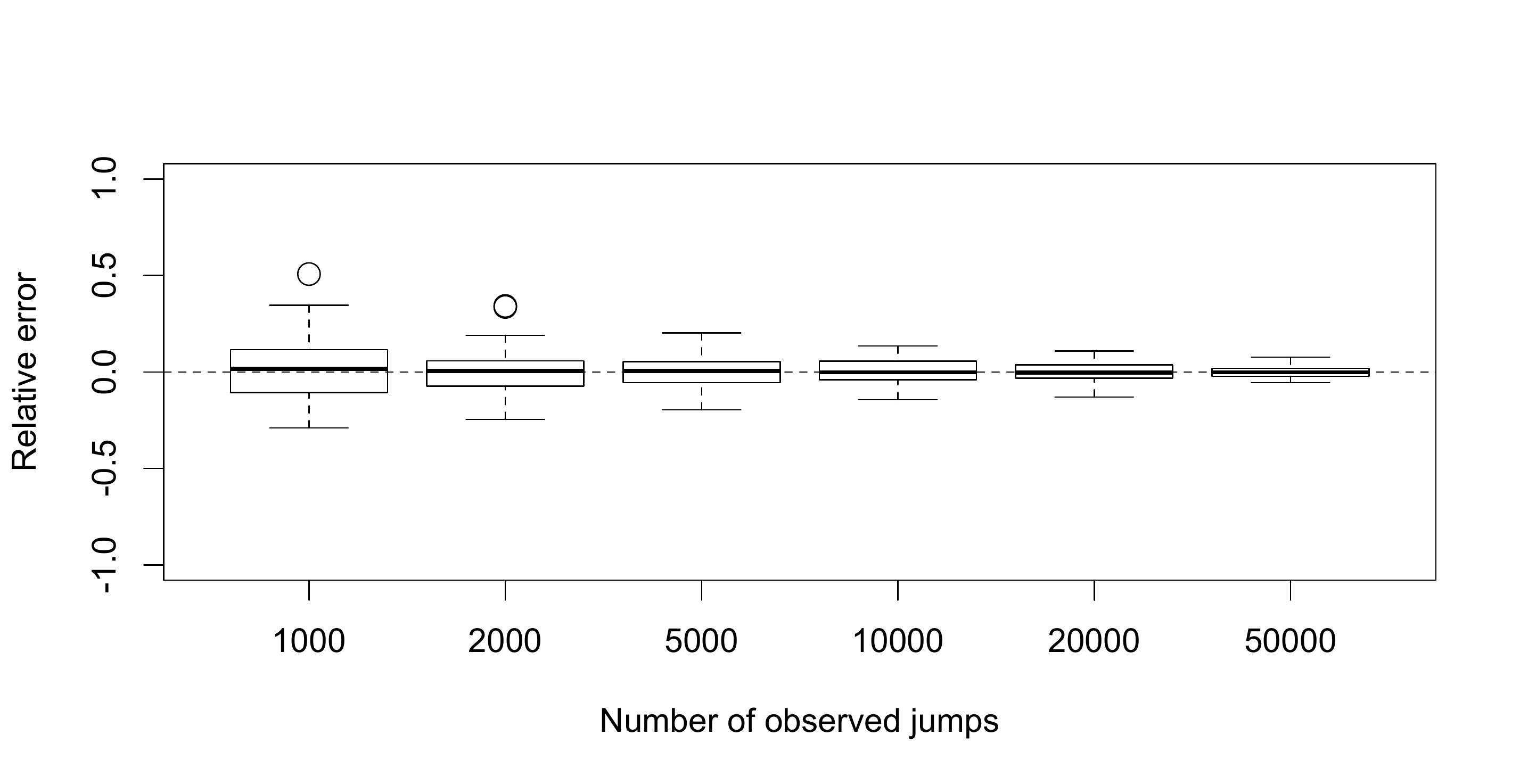} & \includegraphics[width=0.49\textwidth,height=0.28\textwidth]{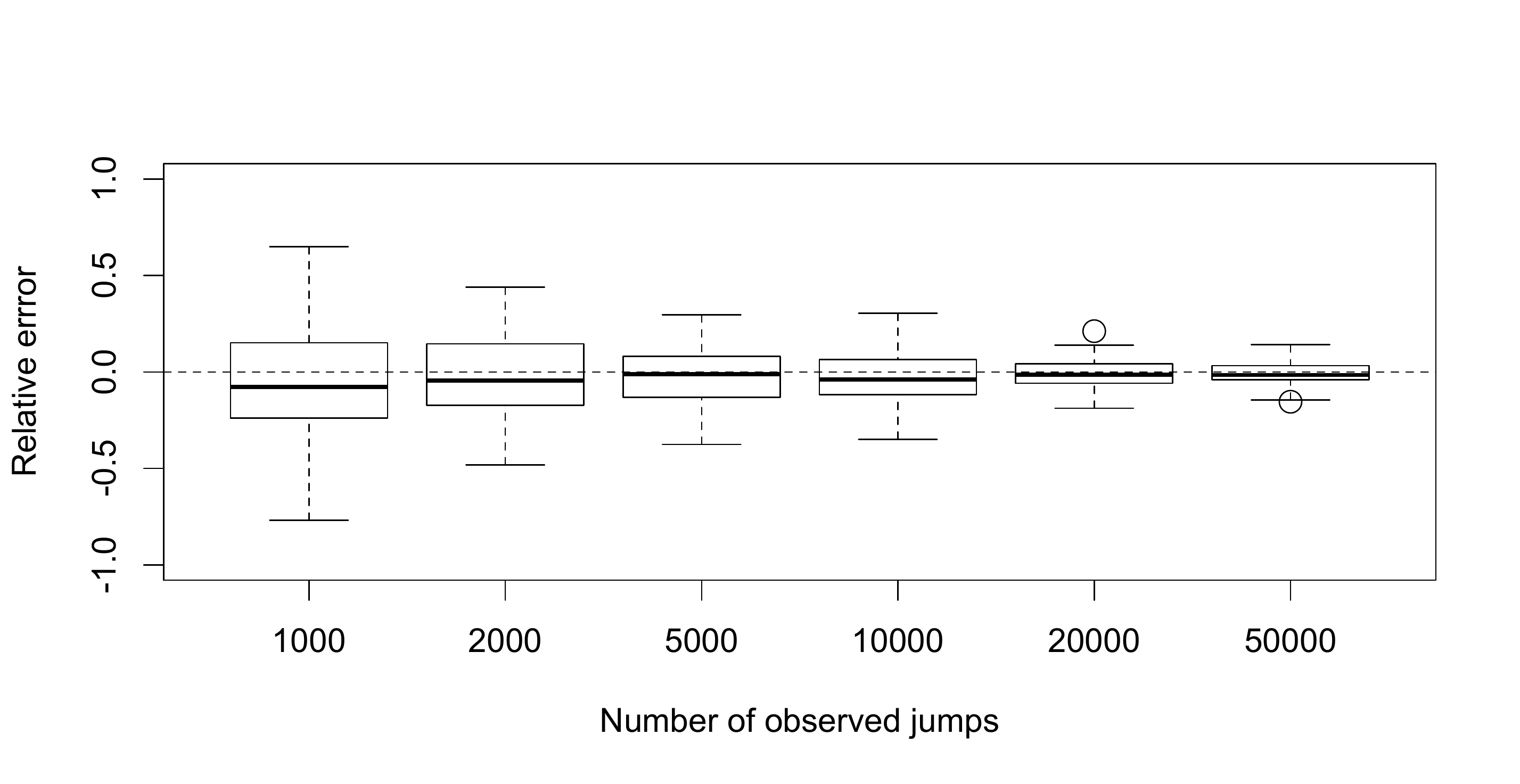}
	\end{tabular}
	\caption{Boxplots over $100$ replicates of the relative error between $\widehat{q}_n(x,y)$ and $q(x,y)$ for $x=1$ and $y=0.5$ (left) and $x=2$ and $y=1$ (right) from different numbers of observed jumps.}
	\label{fig:boxplotsRE}
	\end{figure}

	\begin{figure}[hp]
	\begin{tabular}{cc}
	\includegraphics[width=0.49\textwidth,height=0.3\textwidth]{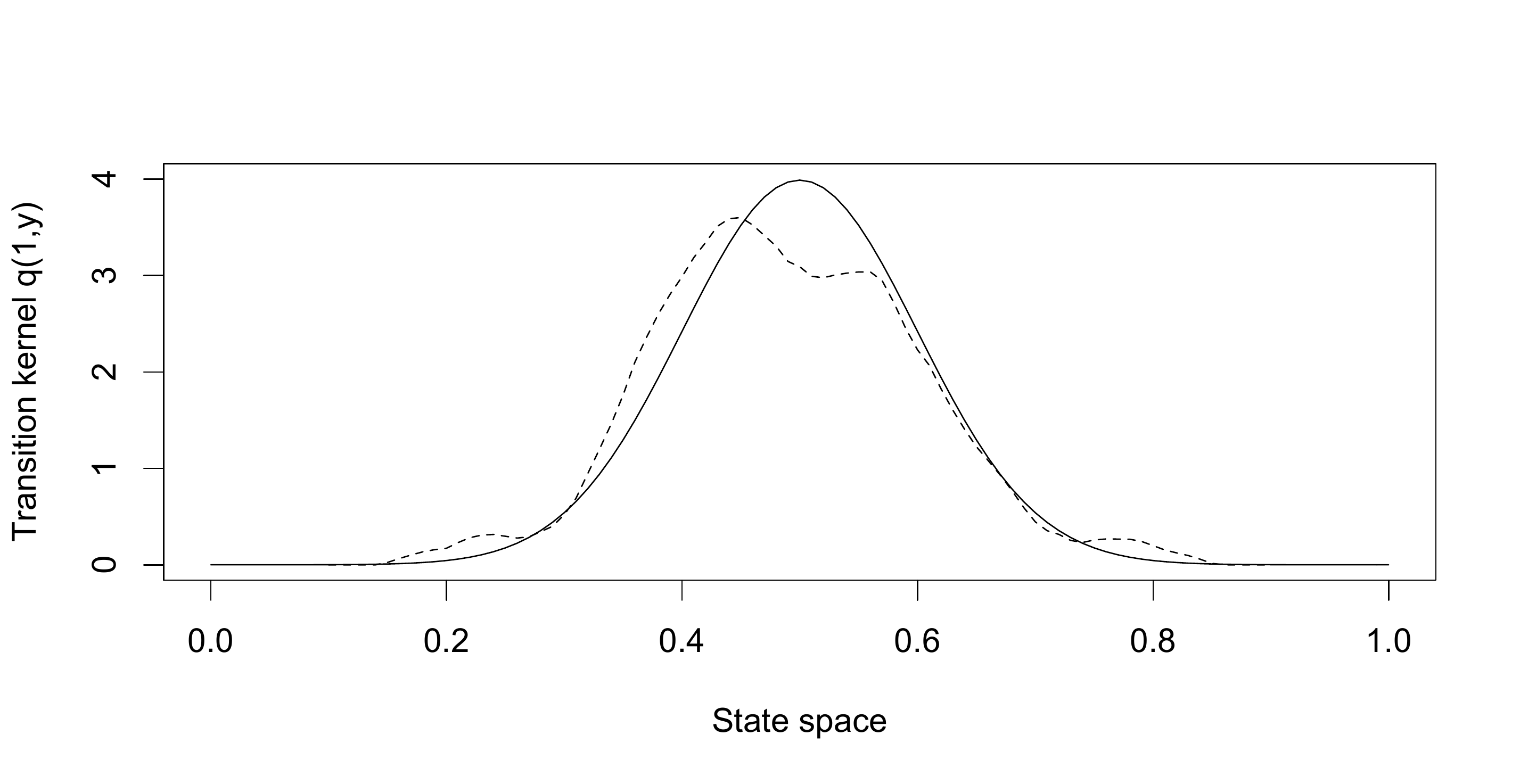} & \includegraphics[width=0.49\textwidth,height=0.3\textwidth]{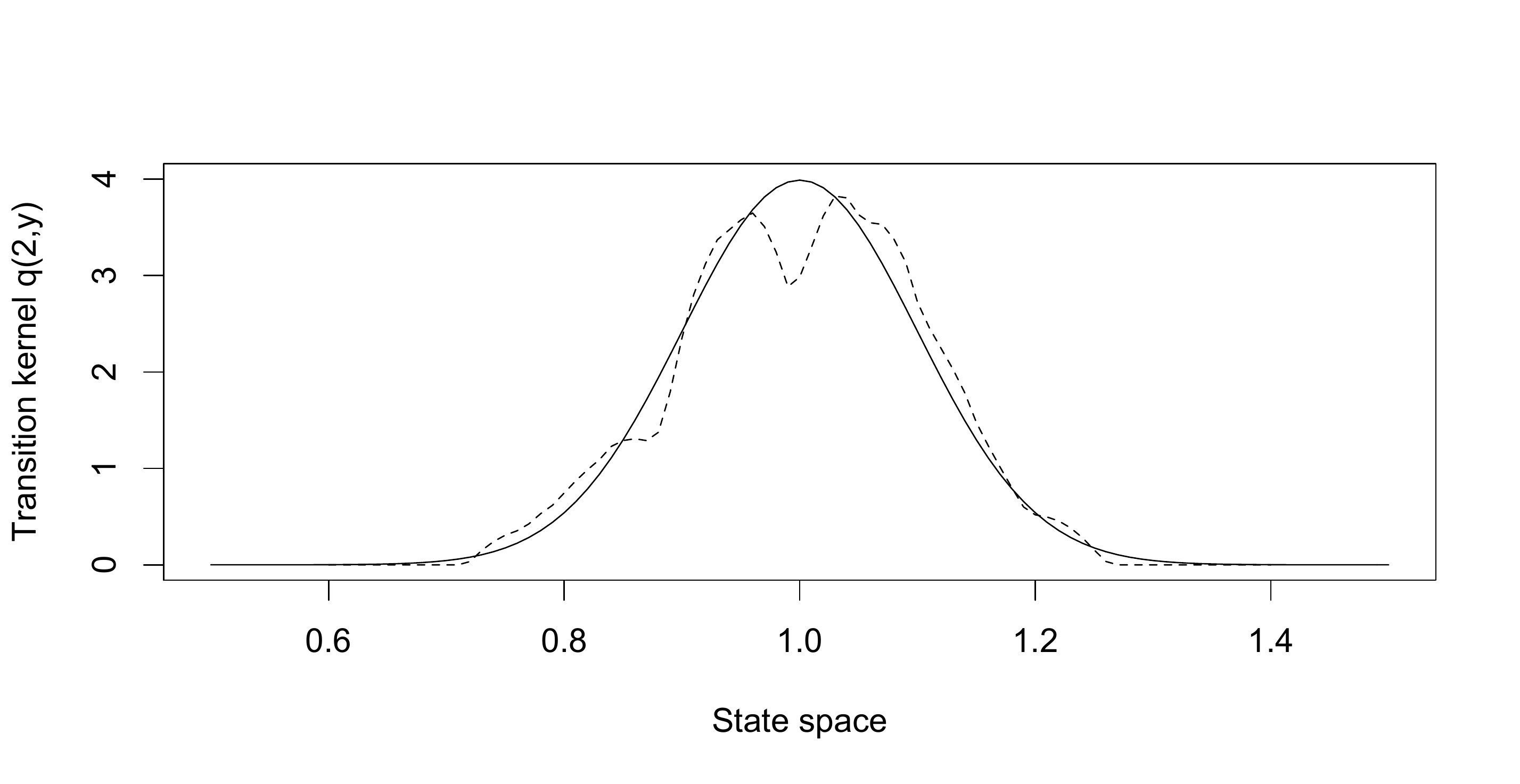} \\
	\includegraphics[width=0.49\textwidth,height=0.3\textwidth]{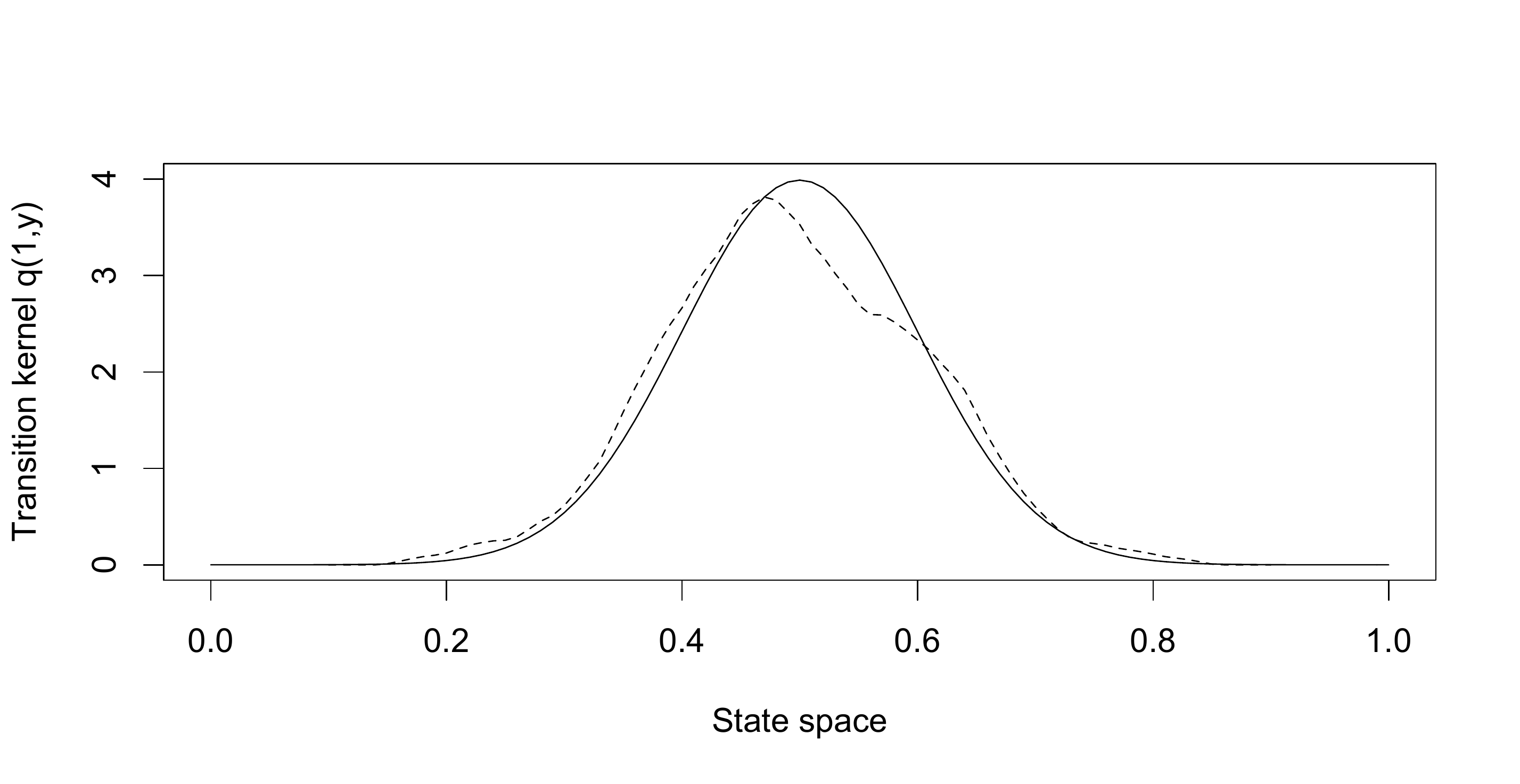} & \includegraphics[width=0.49\textwidth,height=0.3\textwidth]{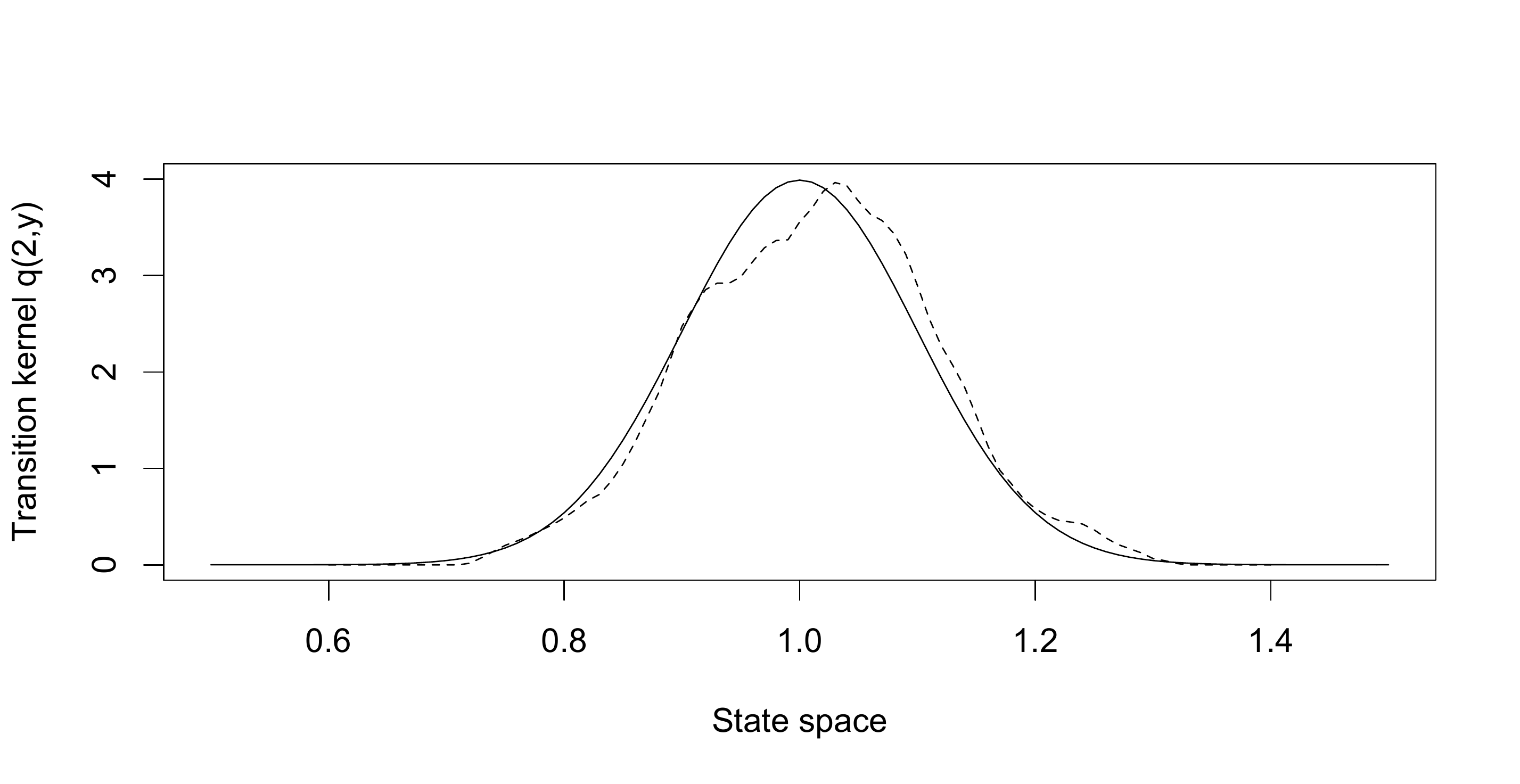}\\
	\includegraphics[width=0.49\textwidth,height=0.3\textwidth]{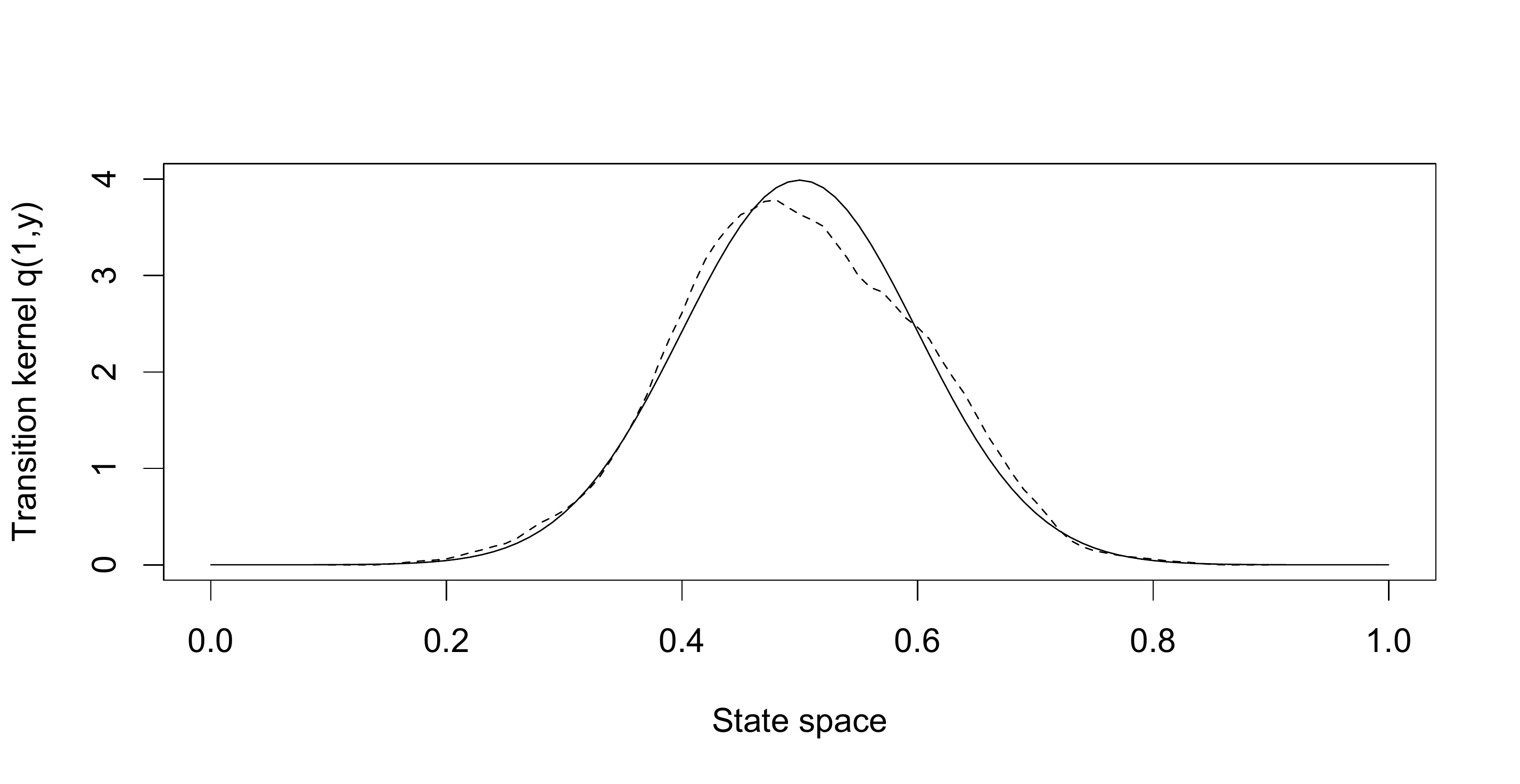} & \includegraphics[width=0.49\textwidth,height=0.3\textwidth]{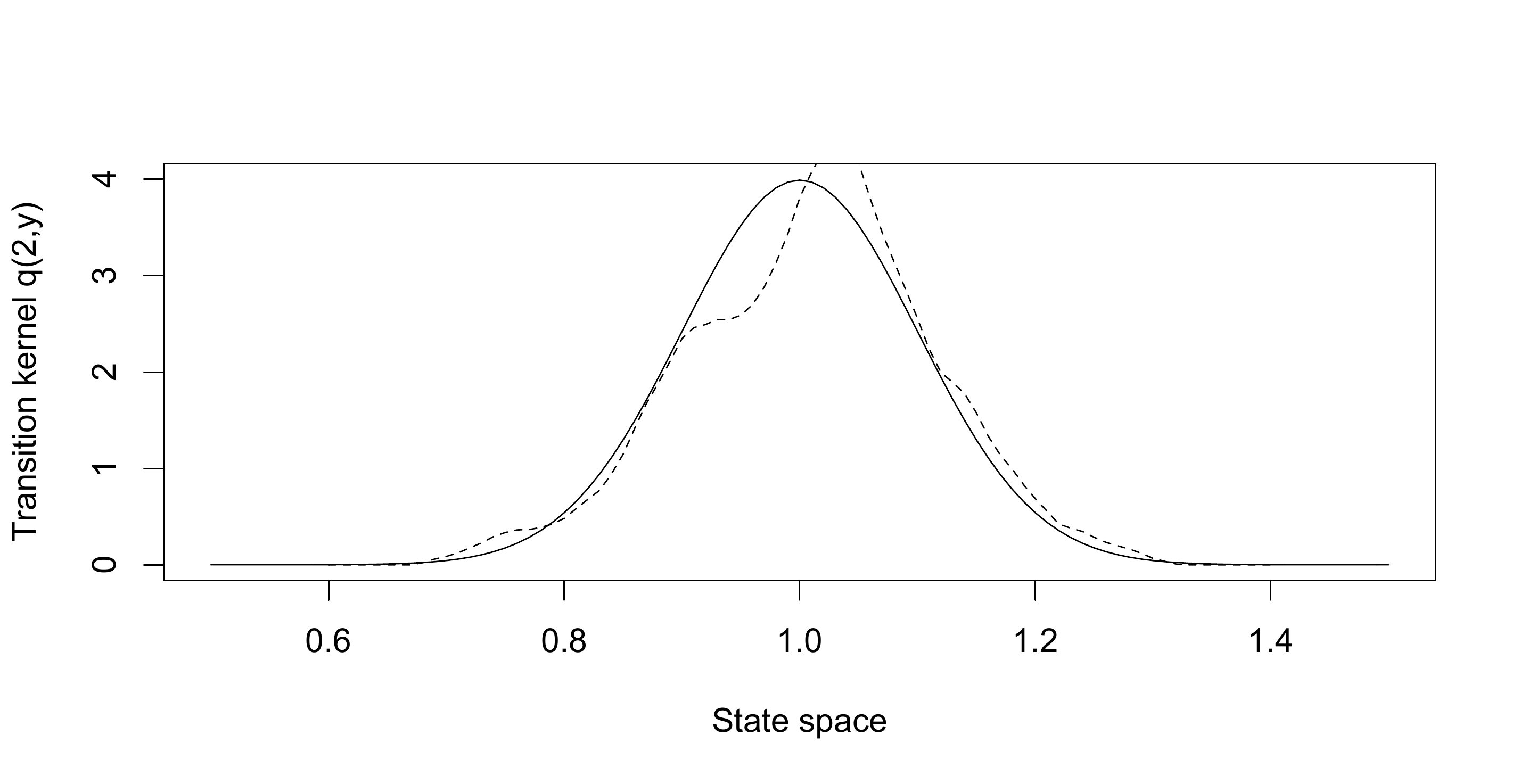}\\
	\includegraphics[width=0.49\textwidth,height=0.3\textwidth]{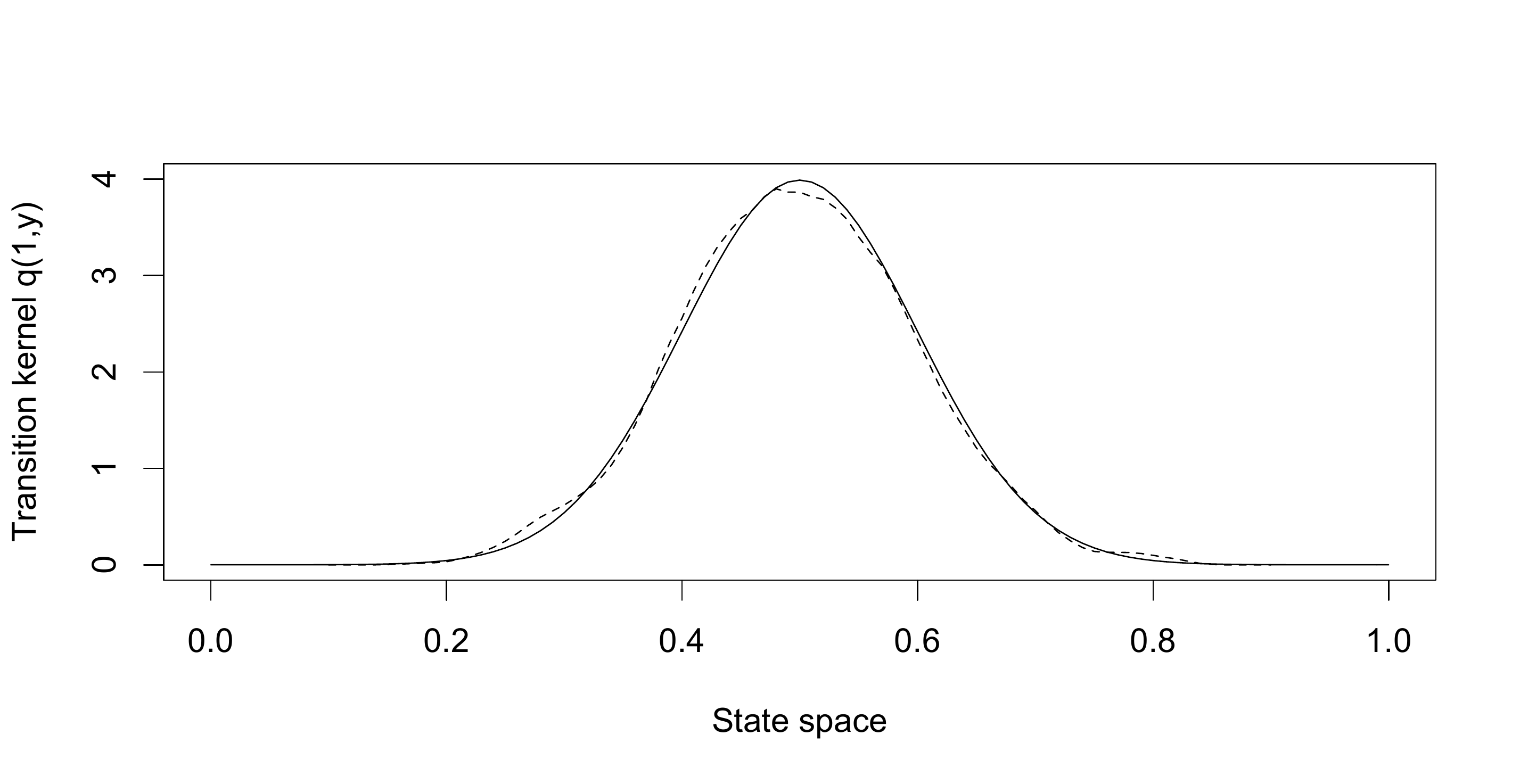} & \includegraphics[width=0.49\textwidth,height=0.3\textwidth]{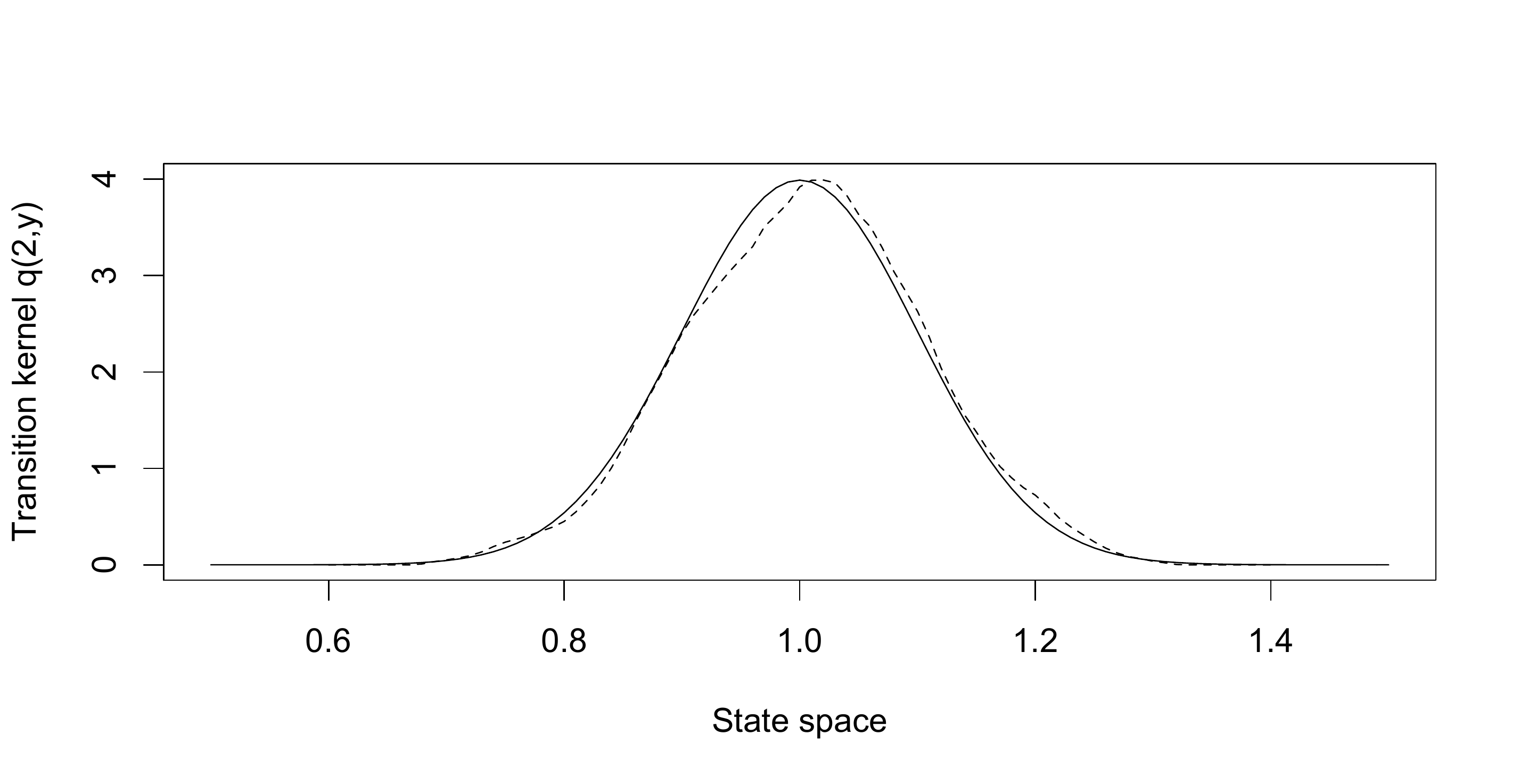}
	\end{tabular}
	\caption{Estimation of the transition density $q(x,\cdot)$ for $x=1$ (left) and $x=2$ (right) from different numbers of observed jumps: $5\,000$, $10\,000$, $20\,000$ and $50\,000$ (from top to bottom).}
	\label{fig:courbes}
	\end{figure}

	\begin{figure}[hp]
	\includegraphics[width=0.7\textwidth,height=0.3\textwidth]{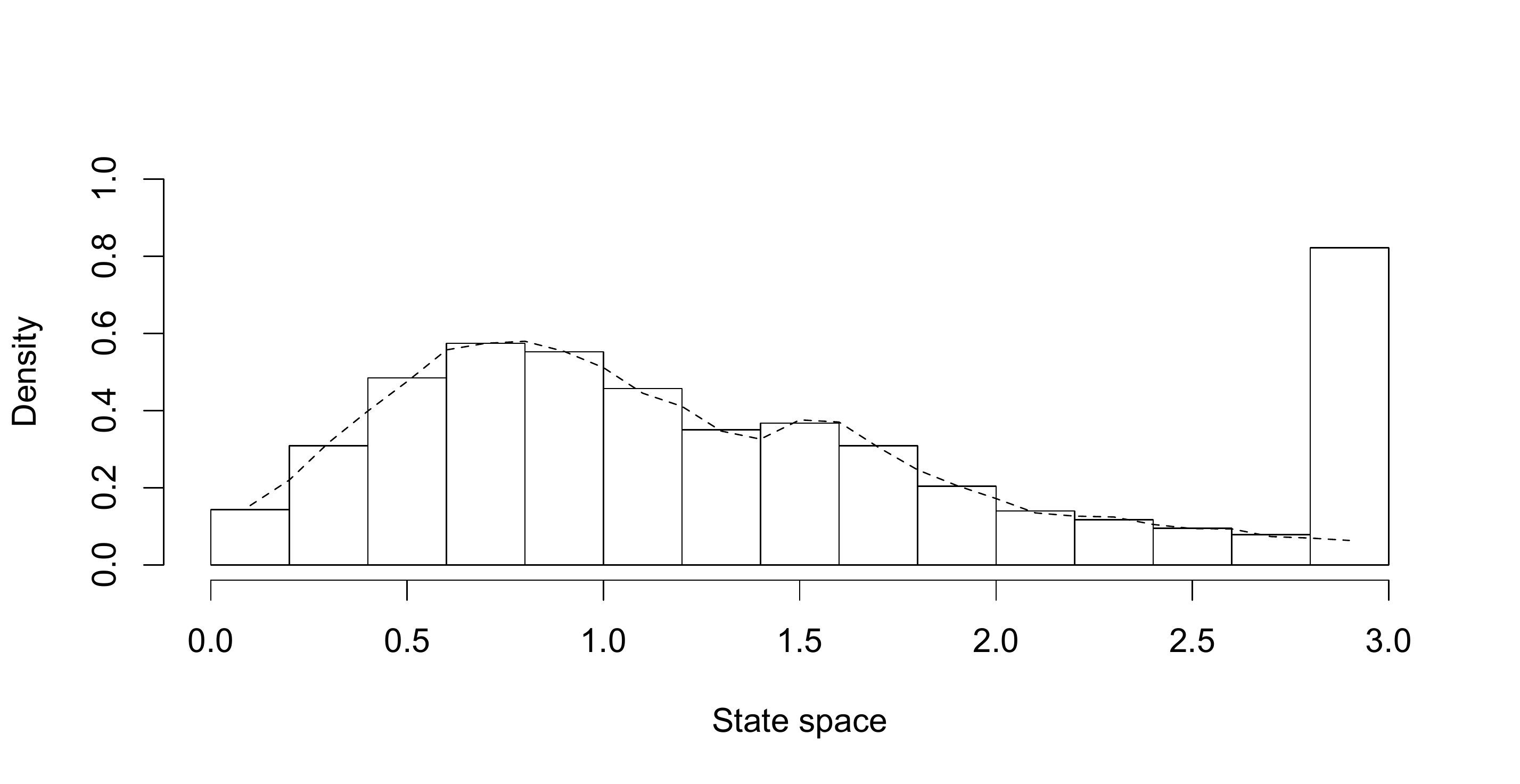}
	\caption{Estimation of the invariant distribution $\pi$ of the Markov chain $(Z_n^-)$. The estimator $\widehat{p}_n(x)$ is drawn in dashed line for $x\in[0.1,2.9]$. The histogram shows the empirical distribution of $(Z_n^-)$.}
	\label{fig:pi}
	\end{figure}

	\begin{figure}[hp]
	\centering
	\includegraphics[width=0.9\textwidth,height=0.3\textwidth]{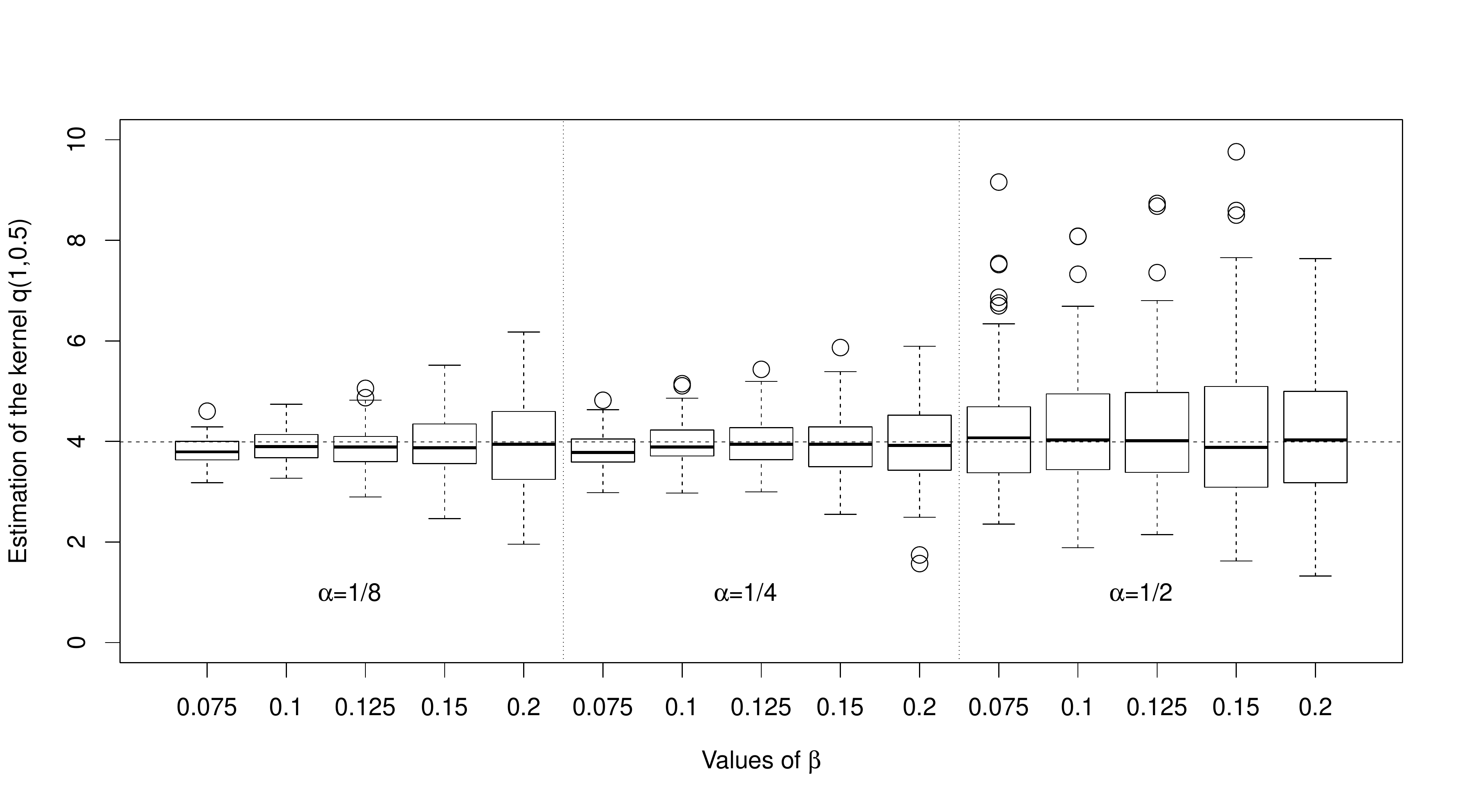}
	\caption{Boxplots over $100$ replicates of the estimator $\widehat{q}_n(1,0.5)$ from $n=10\,000$ observed jumps for different values of the parameters $\alpha$ and $\beta$.}
	\label{fig:parameters}
	\end{figure}

	\begin{figure}[hp]
	\begin{tabular}{cc}
	\includegraphics[width=0.49\textwidth,height=0.34\textwidth]{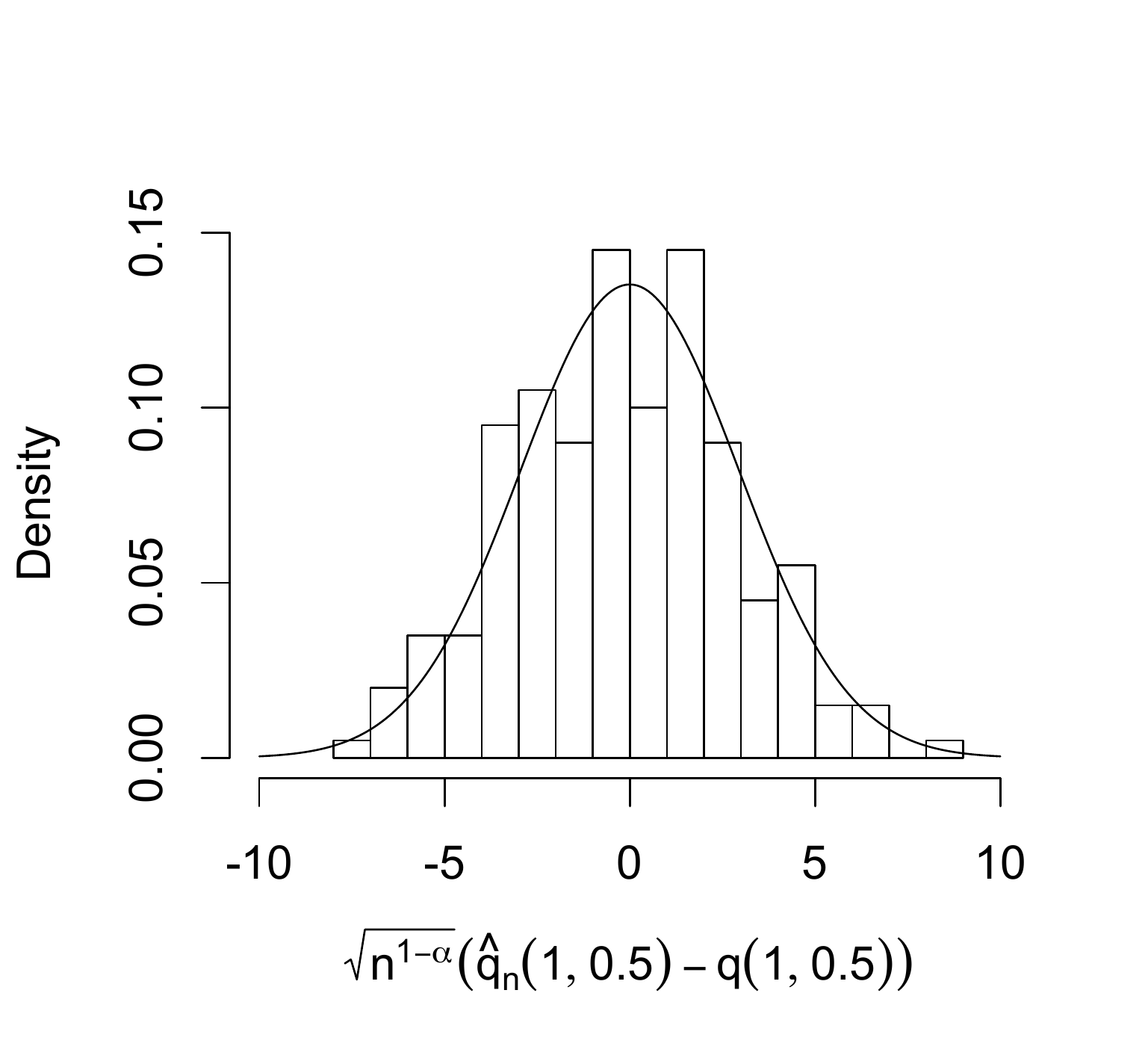} & \includegraphics[width=0.49\textwidth,height=0.34\textwidth]{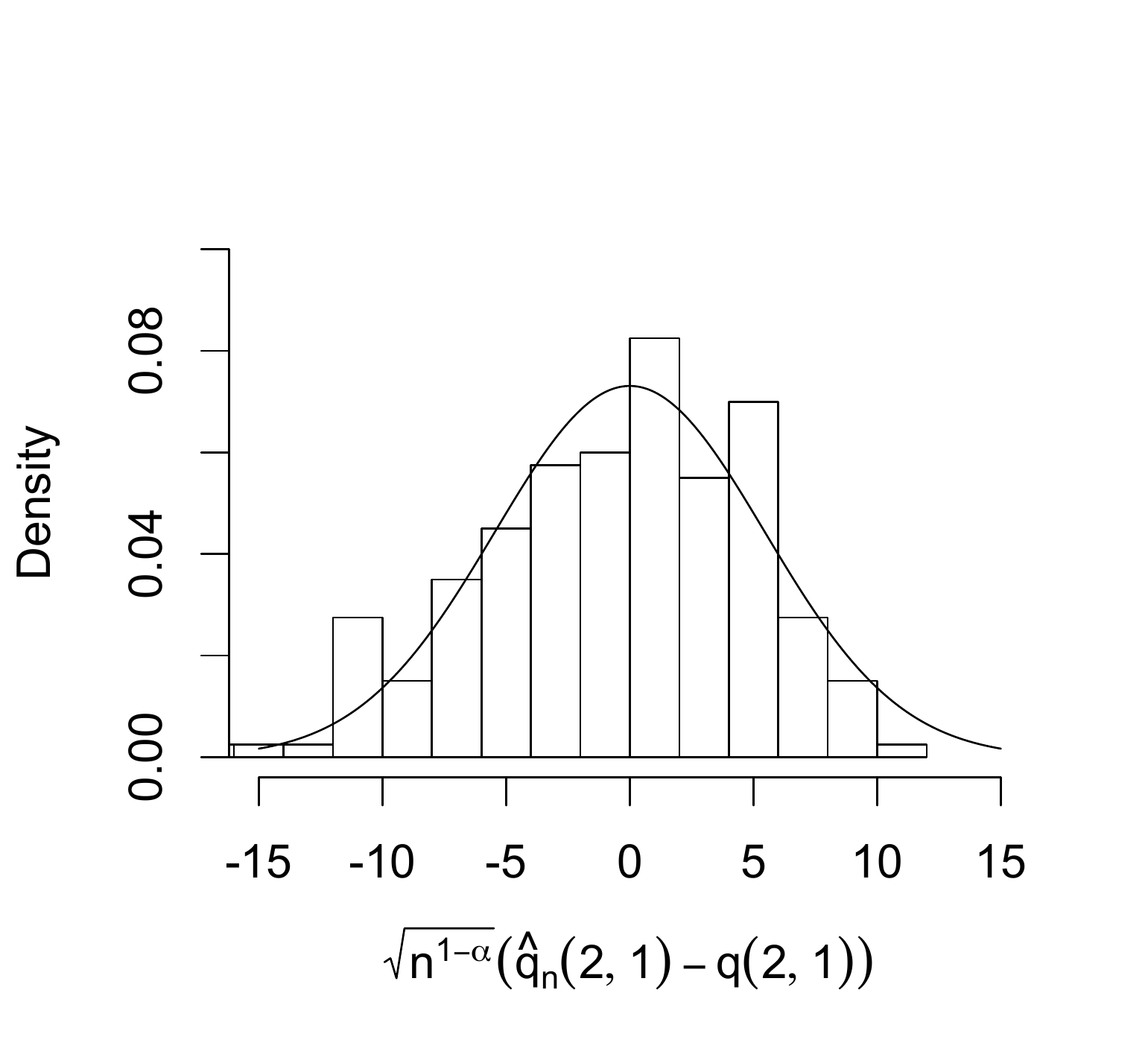}
	\end{tabular}
	\caption{Illustration of the central limit theorem for $\widehat{q}_n(x,y)$ for $x=1$ and $y=0.5$ (left) and $x=2$ and $y=1$ (right) from $50\,000$ observed jumps. The variance of the Gaussian curve has been estimated.}
	\label{fig:TCL}
	\end{figure}

\end{document}

%% file: EstQ_intro.tex
\section{Introduction}

The purpose of this paper is to investigate a nonparametric recursive method for estimating the transition kernel of a piecewise-deterministic Markov process, from only one observation of the process within a long time interval.

Piecewise-deterministic Markov processes (PDMP's) have been introduced in the literature by Davis in \cite{Dav}. They are a general class of non-diffusion stochastic models involving deterministic motion broken up by random jumps, which occur either when the flow reaches the boundary of the state space or in a Poisson-like fashion. The path depends on three local features namely the flow $\Phi$, which controls the deterministic trajectories, the jump rate $\lambda$, which governs the inter-jumping times, and the transition kernel $Q$, which determines the post-jump locations. An appropriate choice of the state space and the main characteristics of the process covers a large variety of stochastic models covering problems in reliability (see \cite{Dav} and \cite{MR2528336}) or in biology (see \cite{subtilin} and \cite{Gen}) for instance. In this context, it appears natural to propose some nonparametric methods for estimating both the characteristics $\lambda$ and $Q$, which control the randomness of the motion. Indeed, the deterministic flow is given by physical equations or deterministic biological models. In \cite{Az12b}, Azaïs \textit{et al.} proposed a kernel method for estimating the conditional probability density function associated with the jump rate $\lambda$, for a non-stationary PDMP defined on a general metric space. This work was based on a generalization of Aalen's multiplicative intensity model and a discretization of the state space. In the present paper, we assume that $Q$ admits a density $q$ with respect to the Lebesgue measure, and we focus on the nonparametric estimation of this function, from the observation of a PDMP within a long time, without assumption of stationarity. In addition, since measured data are often processed sequentially, it is convenient to propose a recursive estimator. To the best of our knowledge, no nonparametric estimation procedures are available in the literature for general PDMP's.

Nonparametric estimation methods for stationary Markov chains have been extensively investigated, beginning with Roussas in \cite{Roussas}. He studied kernel methods for estimating the stationary density and the transition kernel of a Markov chain satisfying the strong Doeblin's condition. Later, Rosenblatt proposed in \cite{Rosenblatt} some results on the bias and the variance of this estimator in a weaker framework. Next, Yakowitz improved in \cite{Yako} the previous asymptotic normality result assuming a Harris-condition. Masry and Györfi in \cite{Masry}, and Basu and Sahoo in \cite{Basu}, have completed this survey. There exists also an extensive literature on nonparametric estimates for non-stationary Markov processes. We do not attempt to present an exhaustive survey on this topic, but refer the interested reader to \cite{Clem,Doukhan, MD,MR1110000,Lacour1,Lacour2,Liebscher} and the references therein. In this new framework, Doukhan and Ghindès have investigated in \cite{Doukhan} a bound of the integrated risks for their estimate. Hern\'andez-Lerma \textit{et al.} in \cite{MR1110000} and Duflo in \cite{MD} made inquiries about recursive methods for estimating the transition kernel or the invariant distribution of a Markov chain. Liebscher gave in \cite{Liebscher} some results under a weaker condition than Doeblin's assumption. More recently, Clémençon in \cite{Clem} proposed a quotient estimator using wavelets and provided the lower bound of the minimax $\mathbf{L}^p$-risk. Lacour suggested in \cite{Lacour2} an estimator by projection with model selection, next she introduced in \cite{Lacour1} an original estimate by inquiring into a new contrast derived from regression framework.

Our investigation and the studies of the literature mentioned before are different and complementary. In this paper, we propose to estimate the transition density $q$ of a PDMP by kernel methods. Nevertheless, we do not observe a Markov chain whose transition distribution is given by $q$. Fortunately, one may write the function of interest as the ratio of two invariant measures: the one of the two components pre-jump location and post-jump location, over the one of the pre-jump location. Indeed, $Q(x,A)$ is defined as the conditional probability that the post-jump location is in $A$, given the path is in $x$ just before the jump. Therefore, we suggest to estimate both these invariant measures in order to provide an estimator of the transition kernel $Q$. A major stumbling block for estimating the invariant law of the pre-jump location is related with the transition kernel of this Markov chain, which may charge the boundary of the state space. As a consequence, the transition kernel, as well as the corresponding invariant distribution, admits a density only on the interior of the state space. The investigated approach for estimating the invariant measure is based on this property of the transition kernel. But the main difficulty appears for analyzing the two-components process pre-jump location, post-jump location. This Markov chain has a special structure, because its invariant distribution admits a density function on the interior of the state space, unlike its transition kernel. Indeed, the pre-jump location is distributed on the curve governed by the deterministic flow initialized by the previous post-jump location. As a consequence, the author have to explore a new method for estimating the two-dimensional invariant measure of interest. The proposed one is more universal, but implies a more restrictive assumption on the shape of the bandwidth.

An intrinsic complication throughout the paper comes from the presence of deterministic jumps, when the path tries to cross the boundary of the state space. Indeed, this induces that the invariant distributions mentioned above may charge a subset with null Lebesgue measure. This important feature has been introduced by Davis in \cite{Dav} and is very attractive for the modeling of a large number of applications. For instance, one may find in \cite{MR1679540} an example of shock models with failures of threshold type. One may also refer the reader to \cite{subtilin}, where the authors develop a PDMP to capture the mechanism behind antibiotic released by a bacteria. Forced jumps are used to model a deterministic switching when the concentration of nutrients rises over a certain threshold.

The paper is organized as follows. Section \ref{def_pdmp} is devoted to the precise formulation of our problem and the presentation of the main results of convergence. The consistency of our nonparametric recursive estimator is stated in Theorem \ref{Q:PDMP:CVps}. A central limit theorem lies in Theorem \ref{Q:PDMP:TCL}. Section \ref{s:simu} deals with numerical considerations for illustrating the asymptotic behavior of our estimate. The strategy and the proofs of the main results are deferred into Appendices \ref{s:s3}, \ref{s:s4} and \ref{s:clt}.



\section{Problem formulation}

This section is devoted to the definition of a piecewise-deterministic Markov process. Moreover, we present also the recursive nonparametric estimator of the transition density $q$ that we consider and our main results of convergence.

\label{def_pdmp}


\subsection{Definition of a PDMP}

We present the definition of a piecewise-deterministic Markov process on $\R^d$, where $d$ is an integer greater or equal to $1$. The process evolves in an open subset $E$ of $\R^d$ equipped with the Euclidean norm $|\cdot|$. The motion is defined by the three local characteristics $(\lambda,Q,\Phi)$.
\begin{itemize}
	\item $\Phi : \R^d\times \mathbf{R}\to \R^d$ is the deterministic flow. It satisfies,
		$$\forall \xi \in \R^d,~\forall s,t\in\mathbf{R},~ \Phi_\xi(t+s) = \Phi_{\Phi_\xi(t)}(s) .$$
		For each $\xi\in E$, $t^+(\xi)$ denotes the deterministic exit time from $E$:
		$$ t^+(\xi) = \inf \{t>0~:~\Phi_\xi(t) \in \partial E\},$$
		with the usual convention $\inf\emptyset = +\infty$.
	\item $\lambda: \R^d\to\mathbf{R}_+$ is the jump rate. It is a measurable function which satisfies,
		$$\forall\xi\in\R^d,~\exists \varepsilon>0,~ \int_0^\varepsilon \lambda\big( \Phi_{\xi}(s) \big) \ud s < +\infty .$$
	\item $Q$ is a Markov kernel on $(\R^d,\mathcal{B}(\R^d))$ which satisfies, for any $\xi\in\R^d$,
	$$Q(\xi,E\setminus\{\xi\})=1 \quad\text{and,}\quad\forall B\in\mathcal{B}(\R^d),~Q(\xi,B) = \int_Bq(\xi,z)\ud z,$$
		where the transition density $q$ is piecewise-continuous.
\end{itemize}

\noindent
There exists a filtered probability space $(\Omega,\mathcal{A},(\mathcal{F}_t),\prob)$, on which a process $(X_t)$ is defined (see \cite{Dav}). Its motion, starting from $x\in E$, can be described as follows. $T_1$ is a positive random variable whose survival function is,
$$ \forall t \geq 0, ~ \mathbf{P}( T_1 > t|X_0=x) = \exp \left(  - \int_0^t \lambda( \Phi_x(s) ) \ud s \right) \mathbf{1}_{\{ 0 \leq t < t^+(x) \}} .$$
One chooses an $E$-valued random variable $Z_1$ according to the distribution $Q(\Phi_{x}(T_1),\cdot)$. Let us remark that the post-jump location depends only on the pre-jump location $\Phi_x(T_1)$. The trajectory between the times $0$ and $T_1$ is given by
\begin{displaymath}
X_t = \left\{ 
\begin{array}{cl}
\Phi_{x}(t) 	& \text{for $0\leq t < T_1$,} \\
Z_1		& \text{for $t=T_1$.}
\end{array}
\right.
\end{displaymath}
Now, starting from $X_{T_1}$, one selects the time $S_2 = T_2-T_1$ and the post-jump location $Z_2$ in a similar way as before, and so on. This gives a strong Markov process with the $T_k$'s as the jump times (with $T_0=0$). One often considers the embedded Markov chain $(Z_n,S_n)$ associated to the process $(X_t)$ with $Z_n=X_{T_n}$, $S_n=T_n-T_{n-1}$ and $S_0=0$. The $Z_n$'s denote the post-jump locations of the process, and the $S_n$'s denote the interarrival times.


In the sequel, we shall consider the discrete-time process $(Z_n^-)$ defined by,
$$\forall n\geq1,~Z_n^- = \Phi_{Z_{n-1}} (S_n) .$$
This sequence is naturally of interest. Indeed, the transition kernel $Q$ describes the transition from $Z_n^-$ to $Z_{n}$. $Z_n^-$ stands for the location of $(X_t)$ just before the $n$th jump. We shall prove that $(Z_n^-)$ is a Markov chain in Lemma \ref{lem:Znmoins}.

\begin{center}\textbf{Some additional notations}\end{center}

Throughout the paper, $f$ and $G$ denote the conditional probability density function and the conditional survival function associated with $\lambda(\Phi_{\cdot}(\cdot))$. Precisely, for all $z\in \R^d$ and $t\geq0$,
\begin{eqnarray*}
G(z,t) &=& \exp\left( - \int_0^t \lambda(\Phi_z(s))\ud s\right) ,\\
f(z,t) &=& \lambda(\Phi_z(t)) G(z,t) .
\end{eqnarray*}
Moreover, $\mathcal{S}$ denotes the conditional distribution of $S_{n+1}$ given $Z_n$, for all integer $n$. For all $z\in E$ and $\Gamma\in\mathcal{B}(\R_+)$, we have
\begin{eqnarray}
\mathcal{S}(z,\Gamma) 	&=& \prob\left( S_{n+1}\in\Gamma | Z_n=z , \sigma(Z_i,S_i\,:\,0\leq i\leq n) \right) \label{expr:S} \\
					&=& \int_{\Gamma\cap[0,t^+(z)[} f(z,s)\ud s ~+~\mathbf{1}_{\Gamma}(t^+(z))G(z,t^+(z)) . \nonumber
\end{eqnarray}
The first term corresponds to random jumps which occur in a Poisson-like fashion, while the second one is associated with deterministic jumps. The relation between $\mathcal{S}$ and the conditional survival function $G$ is given, for all $z\in\R^d$ and $t\geq0$, by $G(z,t) = \mathcal{S}(z,]t,+\infty[)$.

\subsection{Main results}

Our main objective in this paper is to provide a recursive estimator of the transition density $q(x,y)$ for any $(x,y)\in E^2$. The recursive estimator of $q(x,y)$ that we consider may be written as follows,
\begin{equation*}
\widehat{q}_n(x,y)    =   \frac{\displaystyle\sum_{j=1}^{n+1} \frac{1}{w_j^{2d}} K\left(\frac{Z_j^- - x}{w_j}\right) K\left(\frac{Z_j-y}{w_j}\right)}{\displaystyle\sum_{j=1}^{n+1} \frac{1}{v_j^{d}} K\left(\frac{Z_j^- - x}{v_j}\right) },
\end{equation*}
where $w_j = w_1 j^{-\beta}$, $v_j=v_1 j^{-\alpha}$, with $v_1,w_1,\alpha,\beta>0$. In addition, $K$ is a kernel function from $\R^d$ to $\R_+$ satisfying the following conditions,
\begin{enumerate}[(i)]
\item $\text{\normalfont supp}\,K\subset B(0_{\R^d},{\delta})$, where $\delta>0$ and $B(\xi,r)$ stands for the open ball centered at $\xi$ with radius $r$,
\item $K$ is a bounded function.
\end{enumerate}


\noindent
In particular, $\int_{\R^d} K^2(z)\ud z$ is finite. $\tau^2$ denotes this integral in the sequel.

Under the technical conditions given in Assumptions \ref{hyp:mumeasure} and \ref{hyps:regularity01}, we shall state that the Markov chain $(Z_n^-)$ admits a unique invariant measure $\pi$, which has a density $p$ on the interior of the state space (see Corollary \ref{cor:h}). Under these hypotheses, we have a result of consistency.

\begin{theo}\label{Q:PDMP:CVps}
Let us choose $v_1$ and $w_1$ such that $\max(v_1,w_1)\delta<\text{\normalfont dist}(x,\partial E)$. If $p(x)>0$, $\alpha d<1$ and $8\beta d<1$, then,
$$ \widehat{q}_n(x,y) \tops q(x,y) ,$$
when $n$ goes to infinity.
\end{theo}
\begin{proof}
The proof is stated at the end of Appendix \ref{s:s4}.
\end{proof}

\noindent
Under an additional condition presented in Assumption \ref{hyps:clt}, we have the following central limit theorem.

\begin{theo}\label{Q:PDMP:TCL}
Let us choose $v_1$ and $w_1$ such that $\max(v_1,w_1)\delta<\text{\normalfont dist}(x,\partial E)$. If $p(x)>0$,
$$\frac{1}{2+d}<\alpha<\frac{1}{d}\qquad\text{and}\qquad 2(1-\alpha d)<4\beta<\min\left(\frac{1}{2d}\, ,\, \alpha - \frac{1}{2d}\right) ,$$
then, when $n$ goes to infinity,
$$ n^{(1-\alpha d)/2} \big(\widehat{q}_n(x,y) - q(x,y)\big) \stackrel{\mathcal{D}}{\longrightarrow} \mathcal{N}\left(0,\frac{q(x,y)^2\tau^2}{p(x)(1+\alpha d)}\right) .$$
\end{theo}
\begin{proof}
The proof is stated at the end of Appendix \ref{s:clt}.
\end{proof}


\section{Simulation study}
\label{s:simu}

The goal of this section is to illustrate the asymptotic behavior of our recursive estimator via numerical experiments in the one-dimensional case. More precisely, we investigate numerical simulations for an application of PDMP's in a biological context. Our example deals with the behavior of the size of a cell over time, and is a particular growth fragmentation model (see \cite{KRELL}).

We consider a continuous-time process $(X_t)$ which models the size of a cell. This one grows exponentially in time, next the cell splits into two offsprings at a division rate $\lambda$ that depends on its size. We impose that the size can not exceed a certain threshold. In our application, the state space of $(X_t)$ is assumed to be $E=]0,3[$. For any $x\in E$, the deterministic flow $\Phi(x,t)$ is given, for any $t\geq0$, by
$$\Phi(x,t) = x \exp(\tau t),$$
where $\tau=0.9$ in the simulations. We assume that the inter-jumping times are distributed according to the Weibull distribution where the shape parameter is the inverse of the size of the cell. More precisely, for $x\in E$ and $t\geq0$, the conditional density $f(x,t)$ associated with the jump rate $\lambda$ is given by
$$ f(x,t) = \frac{t^{(1-x)/x}\exp\left(-t^{1/x}\right)}{x}.$$
The transition kernel $Q(x,\cdot)$ is chosen to be Gaussian with mean $x/2$, a small variance $\sigma^2$ and truncated to $]x/2\pm\sigma[\cap E$, with $\sigma=10^{-1}$ in our simulations. A trajectory of such a PDMP is given in Figure \ref{fig:trajectoire}.

%

In these numerical experiments, we begin with some investigations of the accuracy of the estimator $\widehat{q}_n(x,y)$, for $(x,y)=(1,0.5)$ and $(x,y)=(2,1)$, from different numbers $n$ of observed jumps. The chosen bandwidth parameters are $v_1=w_1=0.1$, $\alpha=0.125$ and $\beta=0.1$, associated with the Epanechnikov kernel. We present in Figure \ref{fig:boxplots} the boxplots of the estimates over $100$ replicates. The empirical distributions of the associated relative errors are given in Figure \ref{fig:boxplotsRE}. On small-sampled sizes, our procedure is quite unfulfilling. However, for $n$ large enough, our method succeeds in the pointwise estimation of the quantity of interest, especially when $n$ is greater than $10\,000$. In addition, the complete curves $q(x,\cdot)$, with $x=1$ and $x=2$, and their estimates from different numbers of observed jumps are presented in Figure \ref{fig:courbes}. One may observe the convergence of $\widehat{q}_n$ to the transition kernel of interest $q$. From $n=50\,000$ observed jumps, the estimation procedure performs very well.

%

The quality of the estimation is better for $(x,y)=(1,0.5)$ than for $(x,y)=(2,1)$ (see Figures \ref{fig:boxplotsRE} and \ref{fig:courbes}). Thanks to the estimation of the invariant distribution $\pi$ (see Figure \ref{fig:pi}), one may notice that the Markov chain $(Z_n^-)$ is more often around $x=1$ than $x=2$. As a consequence, the number of data for the estimation of $q(x,y)$ is larger for $x=1$ than for $x=2$. The behavior of the limit distribution $\pi$ of $(Z_n^-)$ is a good indicator of the accuracy of the estimator $\widehat{q}_n(x,y)$.
	
We investigate the choice of the bandwidth parameters in Figure \ref{fig:parameters}. We present the boxplots over $100$ replicates of the estimates of  $q(1,0.5)$ from $10\,000$ observed jumps, with different values of the parameters $\alpha$ and $\beta$. In a theoretical point of view, the almost sure convergence holds when $\alpha<1$. However, we observe that $\alpha=1/8$ and $\alpha=1/4$ are better choices than $\alpha=1/2$. According to Figure \ref{fig:parameters}, a good compromise for small numerical bias and variance is $\alpha=1/8$ and $\beta=0.1$.

It is tedious to determine numerically the bandwidth parameters $\alpha$ and $\beta$ for which we observe a central limit theorem for $\widehat{q}_n(x,y)$. An illustration of Theorem \ref{Q:PDMP:TCL} is given in Figure \ref{fig:TCL} from $50\,000$ observed jumps, with $\alpha=0.5$ and $\beta=0.1$.

%% file: EstQ_text01.tex
\appendix

\section{Estimation of the invariant distribution of $(Z_n^-)$}
\label{s:s3}

The main objective of this section is the estimation of the invariant distribution of the Markov chain $(Z_n^-)$. This section is divided into two parts. In the first one, we are interested in the existence and the uniqueness of the invariant distribution of $(Z_n^-)$, and in the properties of its transition kernel $\mathcal{R}$. In the second part, we propose a recursive estimator of the invariant distribution of $(Z_n^-)$ and we investigate its asymptotic behavior.

\subsection{Some properties of $(Z_n^-)$}\label{ss:s31}
In this part, we focus on the process $(Z_n^-)$, which is a Markov chain on $\overline{E}$. We especially investigate its transition kernel $\mathcal{R}$ and the existence of an invariant measure.

	\begin{lem1} \label{lem:Znmoins}
	$(Z_n^-)$ is a Markov chain whose transition kernel $\mathcal{R}$ is given, for all $y\in E$ and $B\in\mathcal{B}(\overline{E})$, by
	\begin{equation}
	\label{expr:R}
	\mathcal{R}(y,B) = \int_E Q(y,\ud z) \mathcal{S}(z,\Phi_z^{-1}(B){\cap\R_+}) ,
	\end{equation}
	where the conditional distribution $\mathcal{S}$ has already been defined by $(\ref{expr:S})$.
	\end{lem1}
	\begin{proof}
	For all integer $n$, by $(\ref{expr:S})$, we have 
	\begin{align*}
	\prob(Z_{n+1}^-&\in B | Z_n=z , Z_n^-,\dots,Z_1^-)\\
	&=\, \prob(S_{n+1}\in\Phi_z^{-1}(B)\cap\R_+ | Z_n=z , Z_n^-,\dots,Z_1^-)\\
	&=\, \esp\left[\esp\big[\mathbf{1}_{\{S_{n+1}\in \Phi_z^{-1}(B)\cap\R_+\}} \big| Z_n=z , \sigma(Z_i,S_i\,:\,0\leq i\leq n)\big] \,\Big| Z_n=z, Z_n^-,\dots,Z_1^- \right] \\
	&=\,  \esp\left[ \mathcal{S}(z,\Phi_z^{-1}(B)\cap\R_+) \big| Z_n=z,Z_n^-,\dots,Z_1^-\right] \\
	&=\,\mathcal{S}(z,\Phi_z^{-1}(B)\cap\R_+).
	\end{align*}
	As a consequence,
	\begin{eqnarray*}
	\prob(Z_{n+1}^- \in B | Z_n^- , \dots,Z_1^-) &=&\int_E \prob(Z_{n+1}^-\in B| Z_n=z , Z_n^-,\dots,Z_1^-) Q(Z_n^-,\ud z) \\
	&=&\int_E \mathcal{S}(z,\Phi_z^{-1}(B)\cap\R_+) Q(Z_n^-,\ud z) ,
	\end{eqnarray*}
	which provides the result.
	\end{proof}

\noindent
We focus on the ergodicity of $(Z_n^-)$ by using Doeblin's assumption.

\begin{hyp}\label{hyp:mumeasure}
We assume that the transition kernel $\mathcal{R}$ satisfies Doeblin's condition (see \cite[page 396]{MandT} for instance), that is, there exist a probability measure $\mu$ on $(\overline{E},\mathcal{B}(\overline{E}))$, a real number $\varepsilon$ and an integer $k$ such that,
\begin{equation}
\label{minor:R}
\forall y\in\overline{E},~\forall B\in\mathcal{B}(\overline{E}),~\mathcal{R}^k(y,B) \geq \varepsilon\,\mu(B) .
\end{equation}
\end{hyp}

\noindent
Under this assumption, one may state that the Markov chain $(Z_n^-)$ is ergodic.

		\begin{prop1} \label{Znmoinsergodic}
		We have the following results.
		\vspace{-0.2cm}
		\begin{itemize}
		\item The Markov chain $(Z_n^-)$ is $\mu$-irreducible, aperiodic and admits a unique invariant measure, which we denote by $\pi$.
		\item There exist $\rho>1$ and $\kappa>0$ such that,
		\begin{equation} \label{eq:convTV}
		\forall n \geq 1,~ \sup_{\xi\in\overline{E}} \|\mathcal{R}^n(\xi,\cdot)-\pi\|_{TV} \leq \kappa \rho^{-n} ,
		\end{equation}		
		where $\|\cdot\|_{TV}$ stands for the total variation norm.
		\item In addition, $(Z_n^-)$ is positive Harris-recurrent.
		\end{itemize}
		\end{prop1}
\begin{proof}
By definition and the inequality $(\ref{minor:R})$, $(Z_n^-)$ is $\mu$-irre\-duci\-ble and aperiodic (see \cite[page 114]{MandT}). Moreover, on the strength of Theorem 16.0.2 of \cite{MandT}, $(Z_n^-)$ admits a unique invariant measure $\pi$ since it is aperiodic and $(\ref{eq:convTV})$ holds. In addition, from Theorem 4.3.3 of \cite{HLL}, $(Z_n^-)$ is positive Harris-recurrent.
\end{proof}


\begin{rem}
The ergodicity of the Markov chain $(Z_n^-)$ is often equivalent to the one of the post-jump locations $(Z_n)$. Doeblin's assumption may be related to the existence of a Foster-Lyapunov's function for instance (see \cite{MandT} for this kind of connection). Furthermore, Assumption \ref{hyp:mumeasure} is satisfied for a PDMP defined on a bounded state space, with Gaussian transitions and a strictly positive jump rate.
\end{rem}

Now, we shall impose some assumptions on the characteristics relative to the flow of the process. Under these new constraints, one may provide a more useful expression of $\mathcal{R}$. In the sequel, for $\xi\in E$, $t^-(\xi)$ denotes the deterministic exit time from $E$ for the reverse flow,
$$t^-(\xi) = \sup\{t<0\,:\, \Phi_x(t)\in\partial E\} ,$$
with the usual convention $\sup\emptyset=-\infty$. Remark that $t^-(\xi)$ is a negative number.

\begin{hyps} \label{hyps:flowww}~
\begin{enumerate}[(i)]
\vspace{-0.2cm}
\item The flow $\Phi$ is assumed to be $\mathcal{C}^1$-smooth. For any $(z,t)\in\R^{d}\times\R$, $D\Phi_z(t)$ is defined by
\begin{equation}
\label{dphi}
D\Phi_z(t) = \left| \det\,\left(      \frac{\partial \Phi_x^{(i)}(t)}{\partial x_j} \right)_{\!\!\!\scriptscriptstyle 1\leq i,j\leq d} \right| .
\end{equation}
\item For any $t\in\R$, $\varphi_t:\R^d\to\R^d$, defined by $\varphi_t(x)=\Phi_x(t)$, is an injective application.
\end{enumerate}
\end{hyps} 

\noindent
A useful expression of the transition kernel $\mathcal{R}$ is stated in the following proposition.

	\begin{prop1} 
	\label{prop:Rdens}
	Let $y\in E$ and $B\in\mathcal{B}(\overline{E})$. We have
	\begin{equation} \label{decompoR}
	\mathcal{R}(y,B) =  \int_{B\cap E} r(y,z)\ud z  ~ + ~ \mathcal{R}(y,B\cap\partial E) ,
	\end{equation}
	where the conditional density function $r$ is given by,
	\begin{equation} \label{expr:r:cool}
	\forall z\in E,~r(y,z) = \int_0^{-t^-(z)} q(y,\Phi_z(-s)) f(\Phi_z(-s),s) D\Phi_z(-s)\ud s .
	\end{equation}
	\end{prop1}
\begin{proof}
Let $B\subset E$. First, we fix $t\geq0$. We define the set $A_t$ by
$$A_t = \{\Phi_\xi(-t)\,:\, \xi\in B\},$$
and we examine the function $\varphi_t:A_t\to B$ defined by $\varphi_t(x)=\Phi_x(t)$, $x\in A_t$. $\varphi_t$ is a $\mathcal{C}^1$-smooth injective application (see Assumptions \ref{hyps:flowww}). Furthermore, for any $z\in B$,
$$\varphi_t(  \Phi_z(-t) ) = \Phi_{\Phi_z(-t)}(t) = z ,$$
with $\Phi_z(-t)\in A_t$ by definition of $A_t$. Consequently, $\varphi_t$ is a $\mathcal{C}^1$-one-to-one correspondence. The inverse function $\varphi_t^{-1}$ is given by $\varphi_t^{-1}(x) = \Phi_x(-t)$, so it is $\mathcal{C}^1$-smooth too. Thus, $\varphi_t$ is a $\mathcal{C}^1$-diffeomorphism from $A_t$ into $B$, which allows us to consider it as a change of  variable. In particular, this shows the relation
\begin{equation}\label{eq:equiv}
 \left(z\in E,\,t\in\mathbf{R}_+, \Phi_z(t)\in B\right)    ~\Leftrightarrow~ \left(t\in\R_+,\,z\in E,\,z\in A_t\right) .
 \end{equation}
Moreover, the Jacobian matrix $\mathbf{J}_{\varphi_t^{-1}}$ of the inverse function $\varphi_t^{-1}$ satisfies,
$$\forall x\in\R^d,~\mathbf{J}_{\varphi_t^{-1}} (x)  =    \left(      \frac{\partial \Phi_x^{(i)}(-t)}{\partial x_j} \right)_{\!\!\!\scriptscriptstyle 1\leq i,j\leq d}    .$$
By $(\ref{expr:S})$ and $(\ref{expr:R})$, we have
$$\mathcal{R}(y,B) = \int_E \int_{\Phi_z^{-1}(B)\cap\R_+} q(y,z)f(z,t)\ud t .$$
Together with $(\ref{eq:equiv})$, we obtain
$$\mathcal{R}(y,B) = \int_{\R_+} \left( \int_{A_t} \mathbf{1}_E(z) f(z,t) q(y,z) \ud z\right) \ud t .$$
By the change of variable $\varphi_t$, we have
\begin{eqnarray*}
\mathcal{R}(y,B)     	&=&		 \int_{\R_+}   \left(   \int_B  \mathbf{1}_E \left( \varphi_t^{-1}(\xi)\right) f\left(   \varphi_t^{-1}(\xi) , t\right)  q\left(y ,  \varphi_t^{-1}(\xi)\right)   \left| \det\, \mathbf{J}_{\varphi_t^{-1}}(\xi) \right| \ud \xi \right) \ud t \\
&=&  \int_{\R_+}   \left(   \int_B  \mathbf{1}_E \left( \Phi_\xi(-t) \right) f\left(   \Phi_\xi(-t) , t\right)  q\left(y ,  \Phi_\xi(-t) \right)  D\Phi_\xi(-t) \ud \xi \right) \ud t ,
\end{eqnarray*}
where $D\Phi_\xi$ is defined by $(\ref{dphi})$. We remark that
$$ \left( \Phi_\xi(-t) \in E \right) \Leftrightarrow \left( 0 \leq t < - t^\star(\xi)\right) ,$$
so $\mathbf{1}_E (\Phi_\xi(-t)) = \mathbf{1}_{\{0 \leq t < - t^\star(\xi)\}}$. By Fubini's theorem, this yields to the expected result.
\end{proof}

\noindent
In the light of this result, one may obtain the following one about the invariant distribution $\pi$ of the Markov chain $(Z_n^-)$: $\pi$ admits a density with respect to the Lebesgue measure in the interior of the state space. In addition, one may exhibit a link between this density and $r$.

	\begin{cor1} \label{cor:h}
	There exists a non-negative function $p$ such that
	\begin{equation}
	\label{eq:pi_dens}
	\forall B\in\mathcal{B}(\overline{E}),~\pi(B) =  \int_{B\cap E} p(x)\ud x ~+~\pi(B\cap\partial E) .
	\end{equation}
	In addition, $p$ is given by the expression,
	\begin{equation}
	\label{eq:exprh}
	\forall x\in E,~p(x) = \int_{\overline{E}} \pi(\ud y) r(y,x) .
	\end{equation}
	\end{cor1}
\begin{proof}
$\mu$ is an irreducibility measure for $\mathcal{R}$. As a consequence and according to Proposition 4.2.2 of \cite{MandT}, the maximal irreducibility measure $\widetilde{\mu}$ is equivalent to the measure $\widetilde{\mu}'$ given for any $B\in\mathcal{B}(\overline{E})$, by
$$\widetilde{\mu}'(B) =  \int_{\overline{E}} \mu(\ud y) \sum_{n\geq0} \mathcal{R}^n(y,B)\frac{1}{2^{n+1}} .$$
$\mathcal{R}$ admits a density on the interior $E$ of the state space $\overline{E}$ of the Markov chain $(Z_n^-)$ (see Proposition \ref{prop:Rdens}). As a consequence, this is the case for $\mathcal{R}^n$, too. Indeed, for any set $B$ such that $\lambda_d(B\cap E)=0$, we have
$$\mathcal{R}^n(y,B\cap E) = \int_{\overline{E}} \mathcal{R}^{n-1}(y,\ud z) \mathcal{R}(z,B\cap E) = 0 .$$
Therefore, $\widetilde{\mu}'(B\cap E)=0$. Finally,
$$\lambda_d(B\cap E) =0~\Rightarrow~\widetilde{\mu}(B\cap E)=0 .$$
Since $\pi$ and the maximal irreducibility measure $\widetilde{\mu}$ are equivalent, $\pi$ admits a density on $E$. Now, we investigate the expression of this density. By Fubini's theorem, we have for any $B\subset E$,
\begin{eqnarray*}
\pi(B) &=&\int_{\overline{E}}\pi(\ud y) \mathcal{R}(y,B) \\
&=& \int_{\overline{E}} \pi(\ud y) \int_B r(y,x)\ud x \\
&=&\int_B \left( \int_{\overline{E}} \pi(\ud y) r(y,x)\right) \ud x .
\end{eqnarray*}
As a consequence, one may identify $p$ with the function $\displaystyle\int_{\overline{E}} \pi(\ud y) r(y,\cdot)$.
\end{proof}

\noindent
One shall see that the regularity of the conditional probability density function $r$ is significant in all the sequel. We state that under additional assumptions, $r$ is Lipschitz.


\begin{hyps} \label{hyps:regularity01}
We assume the following statements.
\vspace{-0.3cm}
\begin{enumerate}[(i)]
\item $t^-$ is a bounded and Lipschitz function, that is,
$$\exists [t^-]_{Lip}>0 ,~\forall x, y\in E,~ \left|t^-(x) - t^-(y)\right|\leq [t^-]_{Lip}|x-y| .$$
\item The flow $\Phi$ is Lipschitz, that is,
$$\exists [\Phi]_{Lip}>0,~\forall x,y \in\R^d,~\forall t\in\R,~\left| \Phi_x(t) - \Phi_y(t)\right| \leq [\Phi]_{Lip} |x-y| .$$
\item $f$ is a bounded and Lipschitz function, that is,
$$\exists [f]_{Lip}>0,~\forall x,y \in\R^d,~\forall t\in\R,~ \left| f(x,t) - f(y,t)\right| \leq [f]_{Lip}|x-y| .$$
\item $q$ is a bounded and Lipschitz function, that is, there exists $[q]_{Lip}>0$ such that, for any $x,y,z\in\R^d$,
$$\textcolor{black}{\left| q(x,y) - q(x,z)\right| \leq [q]_{Lip} |y-z| \quad\text{and}\quad\left|q(x,z)-q(y,z)\right|\leq[q]_{Lip}|x-y|.}$$
\item $D\Phi$ is a bounded and Lipschitz function, that is,
$$\exists [D\Phi]_{Lip}>0,~\forall x,y\in\R^d,~\forall t\in\R,~ \left|D\Phi_x(t)  -  D\Phi_y(t)\right| \leq [D\Phi]_{Lip} |x-y| .$$
\end{enumerate}
\end{hyps}

	\begin{prop1} \label{prop:propr}
	$r$ is a bounded function. Furthermore, there exists a constant $[r]_{Lip}>0$ such that, for any $x\in\overline{E}$, $y\in E$ and $u\in\R^d$ such that $y+u\in E$, we have
	$$|r(x,y+u)-r(x,y)|\leq[r]_{Lip}|u| .$$
	\end{prop1}

\begin{proof}
First, we have from $(\ref{expr:r:cool})$,
$$\|r\|_{\infty} \leq \|t^-\|_{\infty} \|q\|_{\infty} \|f\|_{\infty} \|D\Phi\|_{\infty}.$$
For the second point, we consider the function $\gamma$ defined by,
$$\forall (y,t)\in\R^d\times\R,~\gamma(y,t)= q(x,\Phi_y(-t)) f(\Phi_y(-t),t) D\Phi_y(-t) .$$
This function is Lipschitz as a compound and product of Lipschitz functions (see Assumptions \ref{hyps:regularity01}). $[\gamma]_{Lip}$ stands for its Lipschitz constant, and we have
$$\left|\gamma(y,t) - \gamma(y+u,t)\right| \leq [\gamma]_{Lip} |u| .$$
 In addition, by $(\ref{expr:r:cool})$, the function $r(x,y)$ is given by
$$r(x,y) =  \int_0^{-t^-(y)} \gamma(y,s)\ud s .$$
We suppose that $-t^-(y)\leq -t^-(y+u)$ (recall that $t^-$ is a negative function). We have
$$r(x,y+u)  -  r(x,y)    =   \int_0^{-t^-(y)} \left( \gamma(y+u,s) - \gamma(y,s)\right) \ud s + \int_{-t^-(y)}^{-t^-(y+u)} \gamma(y+u,s)\ud s.$$
As a consequence,
\begin{eqnarray*}
\left| r(x,y+u) - r(x,y)\right| &\leq& \int_0^{\|t^-\|_{\infty}} \left| \gamma(y+u,s) - \gamma(y,s)\right| \ud s \\
&~&~+ ~  \|q\|_{\infty}\|f\|_{\infty}\|D\Phi\|_{\infty} \left| t^-(y) - t^-(y+u)\right|   \\
&\leq& \|t^-\|_{\infty} [\gamma]_{Lip} |u| + [t^-]_{Lip} \|q\|_{\infty}\|f\|_{\infty}\|D\Phi\|_{\infty}  |u|.
\end{eqnarray*}
The obtained inequality for $-t^-(y)>-t^-(y+u)$ is exactly the same one. This achieves the proof.
\end{proof}

\subsection{Estimation of $p$}
\label{ss:s32}

We propose a recursive nonparametric estimator of the function $p$ given in the Corollary \ref{cor:h}. For all integer $n$, the recursive estimator $\widehat{p}_n$ of $p$ that we propose is given for all $x\in E$ by
\begin{equation} \label{eq:exprpn}
\widehat{p}_n(x) = \frac{1}{n} \sum_{j=1}^{n+1} \frac{1}{v_j^d} K\left(\frac{Z_j^- - x}{v_j}\right) ,
\end{equation}
where the bandwith $v_j$ satisfies
$$v_j =  v_1 j^{-\alpha},~\text{with $\alpha>0$} .$$

	\begin{rem} \label{rem:inclusion}
	Let $x\in E$ and $j\geq1$. Since the sequence $(v_n)$ is decreasing, we have
	$$\text{\normalfont supp}\,K\left( \frac{\cdot - x}{v_j}\right) \subset \text{\normalfont supp}\, K\left( \frac{\cdot - x}{v_1}\right)\,\subset\,B(x,v_1{\delta}) .$$
	Thus, if $v_1{\delta} < \text{\normalfont dist}(x,\partial E)$, we have
	$$\text{\normalfont supp}\,K\left(\frac{\cdot-x}{v_j}\right) \subset E .$$
	\end{rem}

\noindent
In the following proposition, we establish the pointwise asymptotic consistency of $\widehat{p}_n$.

	\begin{prop1} \label{prop:pchapeaun}
	Let $x\in E$. One chooses $v_1$ such that $v_1{\delta} <\text{\normalfont dist}(x,\partial E)$ and $\alpha$ such that $\alpha d<1$. Then,
	$$\widehat{p}_n(x) \stackrel{a.s.}{\longrightarrow} p(x) ,$$
	when $n$ goes to infinity.
	\end{prop1}

\begin{proof}
By the expression of $p(x)$ given by $(\ref{eq:exprh})$, the difference $\widehat{p}_n(x)-p(x)$ may be written in the following way,
\begin{eqnarray}
\widehat{p}_n(x) - p(x) &=& \frac{1}{n}\sum_{j=1}^{n+1} \frac{1}{v_j^d} K\left(\frac{Z_j^- - x}{v_j}\right) - \int_{\overline{E}}  r(u,x) \pi(\ud u)\nonumber \\
&=& \frac{1}{nv_1^d} K\left(\frac{Z_1^- - x}{v_1}\right) + \frac{1}{n}M_n + R_n^{(1)} + R_n^{(2)} , \label{eq:decompoh}
\end{eqnarray}
where $M_n$, $R_n^{(1)}$ and $R_n^{(2)}$ are given by
\begin{eqnarray}
M_n &=& \sum_{j=1}^{n} \left[ \frac{1}{v_{j+1}^d}K\left(\frac{Z_{j+1}^- - x}{v_{j+1}}\right) - \int_{\R^d} r(Z_j^- , x+yv_{j+1}) K(y)\ud y \right] , \label{eq:decompo:Mn}\\
R_n^{(1)} &=& \frac{1}{n}\sum_{j=1}^n \int_{\R^d} \left[ r(Z_j^- , x+yv_{j+1})   -   r(Z_j^-,x)\right]K(y)\ud y, \label{eq:decompo:Rn1}\\
R_n^{(2)} &=& \frac{1}{n}\sum_{j=1}^n r(Z_j^- , x)   -   \int_{\overline{E}} r(u,x)\pi(\ud u) . \label{eq:decompo:Rn2}
\end{eqnarray}
The dependency on $x$ is implicit. In $(\ref{eq:decompoh})$, the first term clearly tends to $0$ as $n$ goes to infinity. The sequel of the proof is divided into three parts: in the first one, we show that $R_n^{(2)}$ tends to $0$ by the ergodic theorem. In the second one, we focus on $R_n^{(1)}$ and we prove that this term goes to $0$. Finally, we state that $M_n/n$ tends to $0$ by using the second law of large numbers for martingales. Recall that the Markov chain $(Z_n^-)$ is positive Harris-recurrent with invariant measure $\pi$, according to Proposition \ref{Znmoinsergodic}. Thus, one may apply the ergodic theorem (see for instance Theorem 17.1.7 of \cite{MandT}) and we obtain that $R_n^{(2)}$ almost surely tends to $0$. For $R_n^{(1)}$, we have
\begin{eqnarray}
\left|R_n^{(1)}\right| &\leq & \frac{1}{n}\sum_{j=1}^n \int_{\R^d} \left| r(Z_j^- , x+yv_{j+1}) - r(Z_j^- , x) \right| K(y)\ud y \nonumber\\
&\leq& \frac{1}{n}\sum_{j=1}^n \int_{\R^d} [r]_{Lip} |y| v_{j+1} K(y)\ud y \nonumber\\
&\leq& \frac{1}{n} \sum_{j=1}^n {v_{j+1}} \left( \int_{\R^d} |y|K(y)\ud y\right) [r]_{Lip} . \label{eq:decompo:majorRn1}
\end{eqnarray}
This upper bound tends to $0$ by Cesaro's lemma because the limit of the sequence $(v_n)$ is $0$. Therefore, $R_n^{(1)}$ goes to $0$ as $n$ tends to infinity. Finally, we investigate the term $M_n/n$. First, we show that the process $(M_n)$ is a discrete-time martingale with respect to the filtration $(\mathcal{F}_n)$ defined by,
$$\forall n\geq1,~\mathcal{F}_n = \sigma(Z_1^- , \dots, Z_{n+1}^-) .$$
We have
\begin{equation*}
\esp\left[{M_{n}}|{\mathcal{F}_{n-1}}\right] = {M_{n-1}} + \esp\left[\frac{1}{v_{n+1}^d} K\left(\frac{Z_{n+1}^- - x}{v_{n+1}}\right)\Bigg|Z_n^-\right] - \int_{\R^d}r(Z_n^-,x+yv_{n+1})K(y)\ud y .
\end{equation*}
Thus, we only have to prove that
\begin{equation}
\label{eq:toprove}
\esp\left[\frac{1}{v_{n+1}^d} K\left(\frac{Z_{n+1}^- -x}{v_{n+1}}\right)\Bigg|Z_n^-\right] = \int_{\R^d}r(Z_n^-,x+yv_{n+1})K(y)\ud y .
\end{equation}
We have
\begin{equation*}
 \esp\left[\frac{1}{v_{n+1}^d} K\left(\frac{Z_{n+1}^- - x}{v_{n+1}}\right)\Bigg|Z_n^-\right] = \frac{1}{v_{n+1}^d} \int_{\overline{E}} K\left( \frac{u-x}{v_{n+1}}\right) \mathcal{R}(Z_n^- , \ud u) .
 \end{equation*}
By the assumption on $v_1$ and Remark \ref{rem:inclusion},
\begin{eqnarray*}
\int_{\overline{E}} K\left( \frac{u-x}{v_{n+1}}\right) \mathcal{R}(Z_n^- , \ud u)  & =&\int_E K\left( \frac{u-x}{v_{n+1}}\right) \mathcal{R}(Z_n^-,\ud u) \\
&=& \int_E K\left( \frac{u-x}{v_{n+1}}\right) r(Z_n^- , u)\ud u ,
\end{eqnarray*}
by $(\ref{decompoR})$. Finally, the change of variable $u=yv_{n+1}+x$ states $(\ref{eq:toprove})$. Thus, $(M_n)$ is a martingale. We shall study the asymptotic behavior of its predictable quadratic variation $\langle M\rangle$. A straightforward calculus leads to
\begin{eqnarray*}
(M_n - M_{n-1})^2 &=& \frac{1}{v_{n+1}^{2d}} K^2\left( \frac{Z_{n+1}^- - x}{v_{n+1}}\right) +  \left[ \int_E r(Z_n^- , x+yv_{n+1})K(y)\ud y\right]^2\\
~&~&~{-} \frac{2}{v_{n+1}^d} K\left(\frac{Z_{n+1}^--x}{v_{n+1}}\right) \int_E r(Z_n^- , x+yv_{n+1})K(y)\ud y .
\end{eqnarray*}
Using the method used to show $(\ref{eq:toprove})$, we deduce that
\begin{eqnarray}
\esp\left[ (M_n-M_{n-1})^2|\mathcal{F}_{\textcolor{black}{n-1}}\right] &=&\frac{1}{v_{n+1}^d} \int_E K^2(y) r(Z_n^-,x+yv_{n+1})\ud y \label{eq2:crochetMn}\\
&&~-\left[ \int_E r(Z_n^- , x+y v_{n+1})K(y)\ud y\right]^2 . \nonumber
\end{eqnarray}
As a consequence, there exists a constant $C>0$ such that
$$\langle M\rangle_n \,\leq\, \sum_{j=1}^n \left( \frac{1}{v_{j+1}^d} \|r\|_{\infty}\tau^2 + \|r\|_{\infty}\right)\,\sim\,C\hspace{0.01cm}n^{\alpha d+1}~a.s.$$
when $n$ tends to infinity. By the second law of large numbers for martingales (see Theorem 1.3.15 of \cite{MD}), we have
$$M_n^2 = \mathcal{O}\left(\langle M\rangle_n \ln(\langle M\rangle_n)^{1+\gamma}\right)~a.s.$$
with $\gamma>0$. As a consequence,
$$ \frac{M_n}{n} = \mathcal{O}\left(\sqrt{n^{\alpha d-1}\ln(n^{\alpha d+1})^{1+\gamma}}\right) ~a.s.$$
Thus, $M_n/n$ almost surely tends to $0$ as $n$ goes to infinity if $\alpha d<1$. This achieves the proof.
\end{proof}

%% file: EstQ_text02.tex
\section{Estimation of the invariant distribution of $(Z_n^-,Z_n)$}
\label{s:s4}

In this section, we state that the Markov chain $(Z_n^-,Z_n)$ admits a unique invariant measure. In addition, we are interested in the recursive estimation of this measure. We prove the almost sure convergence of $\widehat{q}_n(x,y)$ given in Theorem \ref{Q:PDMP:CVps} at the end of this section.

\subsection{Some properties of $(Z_n^-,Z_n)$}
\label{ss:s41}

In this part, we focus on the asymptotic behavior of the chain $(Z_n^-,Z_n)$. In all the sequel, $\eta_n$ (respectively $\pi_n$) denotes the distribution of $(Z_n^- , Z_n)$ (resp. $Z_n^-$) for all integer $n$. We have these straightforward relations between $\eta_n$, $Q$ or $q$ and $\pi_n$,
\begin{eqnarray}
\eta_n(A\times B) &=& \int_A Q(z,B)\pi_n(\ud z)  \nonumber\\
&=&\int_{A\times B} q(z,y)\pi_n(\ud z)\ud y . \label{eq:etan}
\end{eqnarray}

	\begin{lem1}
		We have
		$$\lim_{n\to+\infty} \|\eta_n - \eta\|_{TV} = 0 ,$$
		where the limit distribution $\eta$ is defined for all $A\times B\in\mathcal{B}(\overline{E}\times E)$ by
		\begin{equation}\label{eq:eta}\eta(A\times B) = \int_{A\times B} q(z,y) \pi(\ud z)\ud y .\end{equation}
	\end{lem1}
\begin{proof}
Let $g$ be a measurable function bounded by $1$. By virtue of Fubini's theorem, we have
\begin{equation*}
\left| \int_{\overline{E}\times E} g(x,y) \left(\eta_n(\ud x \times \ud y) - \eta(\ud x\times\ud y)\right)\right| \leq\left| \int_{\overline{E}} (\pi_n(\ud x)-\pi(\ud x)) \int_E g(x,y) q(x,y)\ud y \right| ,
\end{equation*}
from $(\ref{eq:etan})$ and $(\ref{eq:eta})$. Thus,
\begin{equation*}
\left| \int_{\overline{E}\times E} g(x,y) \left(\eta_n(\ud x \times \ud y) - \eta(\ud x\times\ud y)\right)\right| \leq \left| \int_{\overline{E}} \widetilde{g}(x) (\pi_n(\ud x)-\pi(\ud x)) \right|,
\end{equation*}
where the function $\widetilde{g}:x\mapsto\int_E g(x,y)q(x,y)\ud y$ is bounded by $1$ since $g$ is bounded by $1$ and $q$ is the conditional density associated with the Markov kernel $Q$. As a consequence,
\begin{equation}\label{eq:convTV-bis}
\left\|\eta_n-\eta\right\|_{TV}\leq\left\|\pi_n-\pi\right\|_{TV} .
\end{equation}
One obtains the expected limit from $(\ref{eq:convTV})$.
\end{proof}

\noindent
In addition, one may prove that $\eta$ admits a density on $E\times E$.

	\begin{lem1}
		There exists a positive function $h$ such that
		$$\eta(A\times B) = \int_{A\times B} h(x,y)\ud x \,\ud y ,$$
		for any $A\times B\in\mathcal{B}(\overline{E}\times E)$ with $A\subset E$. In addition, $h$ is given for all $x,y\in E$ by
		\begin{equation} \label{eq:hpq}
		h(x,y) = p(x) q(x,y).
		\end{equation}
	\end{lem1}
\begin{proof}
From $(\ref{eq:eta})$, we have
\begin{eqnarray*}
\eta(A\times B) &=& \int_{A\times B} q(z,y)\pi(\ud z)\ud y \\
&=& \int_{A\times B} q(z,y) p(z)\ud z \,\ud y,
\end{eqnarray*}
by $(\ref{eq:pi_dens})$ and because $A\subset E$. This achieves the proof.
\end{proof}

\subsection{Estimation of  $h$}
\label{ss:s42}

We propose to estimate the function $h$ by the recursive nonparametric estimator $\widehat{h}_n$ given, for any $(x,y)\in E^2$, by
\begin{equation} \label{eq:exprhn}
\widehat{h}_n(x,y) = \frac{1}{n}\sum_{j=1}^{n+1} \frac{1}{w_j^{2d}} K\left(\frac{Z_j^- - x}{w_j}\right) K\left(\frac{Z_j-y}{w_j}\right) ,
\end{equation}
where the bandwith $w_j$ is given by
$$w_j = w_1 j^{-\beta},~\text{with $\beta>0$}.$$
In the sequel, we are interested in the pointwise convergence of the estimator at a point $(x,y)\in E^2$. We assume that $w_1$ is such that $w_1\delta <\text{dist}(x,\partial E)$, where $\delta$ is the radius of the open ball which contains the support of the kernel function $K$. In this case, Remark \ref{rem:inclusion} is still valid, and we have the following inclusions, for any integer $j$,
\begin{equation} \label{eq:inclusionsK}
\text{supp}\,K\left(\frac{\cdot - x}{w_j}\right) \subset B(x,w_1\delta) \subset E .
\end{equation}
Our main objective is to state in Proposition \ref{prop:hchapeaun} that $\widehat{h}_n(x,y)$ almost surely converges to $h(x,y)$. First, we show that this estimator is asymptotically unbiased (see Proposition \ref{prop:asymptoticunbiased}).

We state some new properties of the distribution measures $\pi_n$ and $\pi$. Let us recall that $\pi_n$ is the law of $Z_n^-$, while $\pi$ is the invariant measure of the Markov chain $(Z_n^-)$.

\begin{lem1} \label{lem:pin}
We have the following statements.
\vspace{-0.2cm}
\begin{itemize}
\item For any integer $n$, $\pi_n$ admits a density function $p_n$ on $E$.
\item $p_n$ is bounded by $\|r\|_{\infty}$ and is an $[r]_{Lip}$-Lipschitz function.
\item $p$ is Lipschitz.
\item For any integer $n$, we have
	\begin{equation}
		\label{eq:convdensity}
		\sup_{x\in E} \left|p_n(x) - p(x)\right| \leq \|r\|_{\infty}\kappa\rho^{-(n-1)} .
	\end{equation}
\end{itemize}
\end{lem1}

\begin{proof}
For the first point, let $B\in\mathcal{B}(\overline{E})$ with $B\subset E$. We have
\begin{eqnarray*}
\pi_n(B)	&=& \int_{\overline{E}}\int_B \mathcal{R}(\xi,\ud y) \pi_{n-1}(\ud \xi)  \\
		&=& \int_B\int_{\overline{E}} r(\xi,y) \pi_{n-1}(\ud\xi) \ud y,
\end{eqnarray*}
where $r$ is the conditional density associated with the kernel $\mathcal{R}$ (see $(\ref{decompoR})$). Thus, one may identify
\begin{equation}\label{eq:pny}
p_n(y)=\int_{\overline{E}} r(\xi,y)\pi_{n-1}(\ud\xi).
\end{equation}
For the second assertion, we have stated in Proposition \ref{prop:propr} that $r$ is a bounded function. As a consequence,
$$\left|p_n(y)\right| \,\leq\, \|r\|_{\infty} \pi_{n-1}(\overline{E}) \,=\,\|r\|_{\infty} .$$
In addition, since $r$ is Lipschitz,
\begin{eqnarray*}
|p_n(y)-p_n(z)| &\leq& \int_{\overline{E}} \left|r(\xi,y)-r(\xi,z)\right|\pi_{n-1}(\ud\xi) \\
~&\leq &\qquad [r]_{Lip} |y-z| \pi_{n-1}(\overline{E}) \qquad=\quad [r]_{Lip} |y-z| .
\end{eqnarray*}
For the third point, $p$ is Lipschitz for the same reason than $p_n$ since $p$ satisfies $(\ref{eq:exprh})$. Finally, for the last point, we have, by $(\ref{eq:exprh})$ and $(\ref{eq:pny})$,
\begin{eqnarray*}
\left|p_n(x) - p(x)\right| &\leq&\int_{\overline{E}} r(y,x) \left|\pi_{n-1}(\ud y)-\pi(\ud y)\right| \\
&\leq& \textcolor{black}{\|r\|_{\infty} \|\pi_{n-1}-\pi\|_{TV}} \\
&\leq &\|r\|_{\infty} \kappa\rho^{-(n-1)} ,
\end{eqnarray*}
by $(\ref{eq:convTV})$. This achieves the proof.
\end{proof}

\noindent
Now, one may state that $\widehat{h}_n(x,y)$ is an asymptotically unbiased estimator of $h(x,y)$.
\begin{prop1}\label{prop:asymptoticunbiased}
When $n$ goes to infinity,
$$\mathbf{E}\left[ \widehat{h}_n(x,y)\right] \to h(x,y) .$$
\end{prop1}

\begin{proof}
We only state that
$$\mathbf{E}\left[\widehat{h}_n(x,y)\right] - \frac{n+1}{n}h(x,y)$$
tends to $0$. We have
\begin{eqnarray*}
\mathbf{E}\left[\widehat{h}_n(x,y)\right] - \frac{n+1}{n}h(x,y)\!\!&=&\!\!\frac{1}{n}\sum_{j=1}^{n+1} \Bigg[ \int_{\overline{E}\times E} \frac{1}{w_j^{2d}}K\left(\frac{u-x}{w_j}\right) K\left(\frac{v-y}{w_j}\right) \pi_j(\ud u) q(u,v)\ud v \\
&~&\qquad\qquad\qquad - \int_{E\times E} h(x,y) K(u) K(v)\ud u\ud v \Bigg].
\end{eqnarray*}
Thanks to $(\ref{eq:inclusionsK})$, one may replace $\overline{E}$ by $E$ in the first integral. As a consequence, one may replace $\pi_j(\ud u)$ by $p_j(u)\ud u$ (see Lemma \ref{lem:pin}). Together with $(\ref{eq:hpq})$ and a change of variables, we obtain
\begin{align*}
\mathbf{E}&\left[\widehat{h}_n(x,y)\right] - \frac{n+1}{n}h(x,y)\\
&= \frac{1}{n}\sum_{j=1}^{n+1} \int_{E\times E} K(u)K(v) \Big( p_j(x+uw_j) q(x+uw_j , y+vw_j) - p(x)q(x,y)\Big) \ud u\,\ud v .
\end{align*}
Furthermore, since $p_j$ is $[r]_{Lip}$-Lipschitz and bounded by $\|r\|_{\infty}$ in the light of Lemma \ref{lem:pin}, an elementary calculus leads to
\textcolor{black}{\begin{eqnarray*}
\big| p_j(x+uw_j)\!\!\!\!&&\!\!\!\!\!\!\!\!\!\!q(x+uw_j , y+vw_j) - p(x)q(x,y)\big|\\
&\leq& \|r\|_{\infty}[q]_{Lip} w_j \left(|u|+|v|\right)+ \|q\|_{\infty} [r]_{Lip} |u|w_j + \|q\|_{\infty} \left|p_j(x) - p(x)\right| \\
&\leq& w_j\Big( |u|   \big( \|r\|_{\infty}[q]_{Lip}+\|q\|_{\infty} [r]_{Lip}\big) + |v| \|r\|_{\infty}[q]_{Lip}\Big) + \|q\|_{\infty} \|r\|_{\infty} \kappa\rho^{-(j-1)},
\end{eqnarray*}}
together with $(\ref{eq:convdensity})$. Finally, we obtain
\begin{align}
\Big|\mathbf{E}&\left[\widehat{h}_n(x,y)\right] - \frac{n+1}{n}h(x,y)\Big| \nonumber\\
&\leq \frac{\left(\textcolor{black}{2}\|r\|_{\infty}[q]_{Lip}+\|q\|_{\infty}[r]_{Lip}\right)\int_E K(u)|u|\ud u}{n}\sum_{j=1}^{n+1} w_j + \frac{\|r\|_{\infty}\|q\|_{\infty}\kappa}{n} \sum_{j=1}^{n+1}\rho^{-(j-1)}\nonumber \\
&\leq \frac{\left(\textcolor{black}{2}\|r\|_{\infty}[q]_{Lip}+\|q\|_{\infty}[r]_{Lip}\right)\int_E K(u)|u|\ud u}{n}\sum_{j=1}^{n+1} w_j + \frac{\|r\|_{\infty}\|q\|_{\infty}\kappa}{n(1-\rho^{-1})}, \label{eq:biais:maj}
\end{align}
which tends to $0$ by Cesaro's lemma.
\end{proof}

In the following, we are interested in some properties of the discrete-time process
\begin{equation}\label{eq:exprAn}
(A_n) = \left( \frac{1}{w_n^{2d}} K\left(\frac{Z_n^- - x}{w_n}\right) K\left(\frac{Z_n-y}{w_n}\right)\right) ,
\end{equation}
which naturally appears in the study of the estimator $\widehat{h}_n(x,y)$. In particular, we propose to investigate its autocovariance function. On the strength of this result, we will establish the asymptotic behavior of the variance of $\widehat{h}_n(x,y)$.

\begin{prop1} \label{prop:varcov}
There exist two constants $B$ and $b>1$ such that, for any integers $n\geq k$,
\begin{equation}
\big|\cov(A_k,A_n)\big| \leq \frac{\|K\|_{\infty}^4 B}{w_n^{4d}} b^{k-n}\left(1+b^{-k}\right). \label{eq:cov}
\end{equation}
In particular, one obtains by taking $k=n$ and using that $b^{-n}\leq 1$,
\begin{equation}\label{eq:varianceAn}
\var(A_n) \leq \frac{2\|K\|_{\infty}^4B}{w_n^{4d}}.
\end{equation}
\end{prop1}
\begin{proof}
We have
\begin{align*}
&\cov(A_k,A_n) \\
&\,=\!\frac{\|K\|_{\infty}^4}{w_n^{2d}w_k^{2d}}\cov\!\left(\!\frac{1}{\|K\|_{\infty}^2}\!K\!\left(\frac{Z_k^- - x}{w_k}\right)\!K\!\left(\frac{Z_k-y}{w_k}\right),
\frac{1}{\|K\|_{\infty}^2}\!K\!\left(\frac{Z_n^- - x}{w_n}\right) K\!\left(\frac{Z_n-y}{w_n}\right)\!\right)\! ,
\end{align*}
where both the components in the covariance are bounded by $1$. We apply Theorem 16.1.5 of \cite{MandT} with $V=1$, $\Phi=(Z^-,Z)$,
$$ g= \frac{1}{\|K\|_{\infty}^2} K\left(\frac{\cdot-x}{w_n}\right) K\left(\frac{\cdot-y}{w_n}\right) \quad\text{and}\quad h= \frac{1}{\|K\|_{\infty}^2} K\left(\frac{\cdot-x}{w_k}\right) K\left(\frac{\cdot-y}{w_k}\right).$$
The conditions of the theorem are satisfied by $(\ref{eq:convTV})$ and $(\ref{eq:convTV-bis})$. We obtain
$$\left|\cov(A_k,A_n)\right| \leq \frac{\|K\|_{\infty}^4}{w_n^{2d}w_k^{2d}} B b^{k-n}\left(1+b^{-k}\right) .$$
Together with $1/w_k^{2d} \leq 1/w_n^{2d}$, this shows $(\ref{eq:cov})$.
\end{proof}


\noindent
In the following result, we give a bound for the variance of $\widehat{h}_n(x,y)$. It is a corollary of Proposition \ref{prop:varcov}.
\begin{cor1} \label{cor:varHn}
Let $n$ be an integer. We have
\begin{equation}
\label{eq:varhn}
\var\left(\widehat{h}_n(x,y)\right) \leq \frac{8}{n w_{n+1}^{4d}} \frac{\|K\|_{\infty}^4 B}{1-b^{-1}} .
\end{equation}
As a consequence, this variance goes to $0$ when $4\beta d<1$ (recall that $w_n=w_1 n^{-\beta}$).
\end{cor1}
\begin{proof}
This inequality is a consequence of $(\ref{eq:cov})$ stated in Proposition \ref{prop:varcov}. Indeed, in light of the expressions of $A_n$ $(\ref{eq:exprAn})$ and $\widehat{h}_n(x,y)$ $(\ref{eq:exprhn})$, we have
\begin{eqnarray*}
\var\left(\widehat{h}_n(x,y)\right) &=& \frac{2}{n^2} \sum_{k=1}^{n+1}\sum_{l=k}^{n+1}\cov(A_l,A_k) \\
&\leq& \frac{2}{n^2} \|K\|_{\infty}^4 {B}    \sum_{k=1}^{n+1}\sum_{l=k}^{n+1} {b}^{k-l}\left(1+{b}^{-k}\right) w_l^{-4d} .
\end{eqnarray*}
Using that ${b}^{-k}\leq 1$ and $w_l^{-4d}\leq w_{n+1}^{-4d}$,
\begin{eqnarray*}
\var\left(\widehat{h}_n(x,y)\right)  &\leq& \frac{4}{n^2 w_{n+1}^{4d}} \|K\|_{\infty}^4   {B} \sum_{k=1}^{n+1}{b}^k\sum_{l=k}^{n+1}{b}^{-l} \\
&\leq &  \frac{4}{n^2 w_{n+1}^{4d}} \|K\|_{\infty}^4\  {B} \sum_{k=1}^{n+1} {b}^k \frac{{b}^{-k}}{1-{b}^{-1}} \\
&\leq& \frac{8}{n w_{n+1}^{4d}} \frac{\|K\|_{\infty}^4B }{1-b^{-1}},
\end{eqnarray*}
with $(n+1)/n\leq 2$.
\end{proof}

\noindent
Now, one may state the consistency of our estimator of $h(x,y)$.
\begin{prop1}\label{prop:hchapeaun}
	Let $(x,y)\in E^2$. One chooses $w_1 \delta <\text{\normalfont dist}(x,\partial E)$ and $8\beta d<1$. Then,
	$$\widehat{h}_n(x,y) \stackrel{a.s.}{\longrightarrow} h(x,y) ,$$
	when $n$ goes to infinity.
\end{prop1}

\begin{proof}
According to Proposition \ref{prop:asymptoticunbiased}, we only have to prove that
\begin{equation}
\label{eq:defYn}
Y_n=\left(\widehat{h}_n(x,y)-\esp\left[\widehat{h}_n(x,y)\right]\right)^2
\end{equation}
almost surely converges to $0$. In the sequel of the proof, we establish that there exists a random variable $Y$ such that $Y_n\tops Y$. Since the sequence $(Y_n)$ tends to $0$ in $\mathbf{L}^1$ (remark that $\esp[Y_n]=\var\left(\widehat{h}_n(x,y)\right)$, together with Corollary \ref{cor:varHn}), $Y=0~a.s.$ and it induces the expected result. In order to show the almost sure convergence of the sequence $(Y_n)$, we use Van Ryzin's lemma (see \cite{MR0258172}). In light of this result, if the sequence $(Y_n)$ satisfies the following conditions,
\vspace{-0.2cm}
\begin{enumerate}[(i)]
\item $Y_n\geq0~a.s.$,
\item $\esp[Y_1]<+\infty$,
\item $\esp\left[Y_{n+1}|\mathcal{S}_n\right] \leq Y_n + Y'_n~a.s.$, where $\mathcal{S}_n=\sigma(Y_1,\dots,Y_n)$ and $Y'_n$ is $\mathcal{S}_n$-measurable,
\item $\sum_{n\geq1} \esp\left[|Y'_n|\right] <+\infty$,
\end{enumerate}
\vspace{-0.2cm}
then $Y_n\tops Y$. In our context, points (i) and (ii) are unquestionably satisfied. Let us define
\begin{equation}\label{eq:defVn}
V_n= \frac{1}{n}\sum_{k=1}^{n+1} \big(   A_k -\esp[A_k]\big) .
\end{equation}
By $(\ref{eq:exprhn})$, $(\ref{eq:exprAn})$ and $(\ref{eq:defYn})$, $V_n^2=Y_n$ and we have the recurrence relation,
$$V_n = \frac{A_{n+1}-\esp[A_{n+1}] + (n-1)V_{n-1}}{n} .$$
By squaring, we obtain
\begin{eqnarray*}
Y_n &=& \left(\frac{n-1}{n}\right)^2 Y_{n-1} + \frac{ \big(A_{n+1}-\esp[A_{n+1}]\big)^2 + 2(n-1)V_{n-1}\big(A_{n+1}-\esp[A_{n+1}]\big)}{n^2} \\
&\leq & Y_{n-1} +  \frac{ \big(A_{n+1}-\esp[A_{n+1}]\big)^2 + 2(n-1)V_{n-1}\big(A_{n+1}-\esp[A_{n+1}]\big)}{n^2} .
\end{eqnarray*}
Finally, $\esp[Y_n|\mathcal{S}_{n-1}] \leq Y_{n-1} + Y'_{n-1}$, where 
\begin{equation}
\label{eq:Yprimen}
Y'_{n-1} = \frac{1}{n^2} \esp\left[  \big(A_{n+1}-\esp[A_{n+1}]\big)^2 + 2(n-1)V_{n-1}\big(A_{n+1}-\esp[A_{n+1}]\big) \Big| \mathcal{S}_{n-1}\right] .
\end{equation}
Thus, (iii) is checked. Ultimately, we have to verify (iv). By $(\ref{eq:Yprimen})$ and Cauchy-Schwarz inequality,
\begin{eqnarray*}
\esp\left[|Y_{n-1}'|\right] &\leq& \frac{1}{n^2}\var(A_{n+1}) + \frac{2(n-1)}{n^2} \esp\left[\big|V_{n-1}\big|\,\big|A_{n+1}-\esp[A_{n+1}]\big|\right]  \\
&\leq & \frac{1}{n^2}\var(A_{n+1}) +  \frac{2}{n}  \sqrt{\esp[Y_{n-1}] \var(A_{n+1})} \\
&\leq & \frac{1}{n^2}\var(A_{n+1}) + \frac{2}{n} \sqrt{\var\left(\widehat{h}_{n-1}(x,y)\right) \var(A_{n+1})} .
\end{eqnarray*} 
Thus, by $(\ref{eq:varianceAn})$ and $(\ref{eq:varhn})$, there exist two numbers $c_1$ and $c_2$ such that
\begin{eqnarray*}
\esp\left[|Y_{n-1}'|\right]  &\leq & \frac{c_1}{n^2 w_n^{4d}} + \frac{c_2}{n^{3/2} w_n^{4d}} \\
&\leq & \frac{c_1}{w_1 n^{2-4\beta d}} + \frac{c_2}{w_1 n^{3/2-4\beta d}} .
\end{eqnarray*}
As a consequence, $\sum \esp[|Y_n'|]$ is a convergent series for $8\beta d<1$.
\end{proof}

Now, we give the proof of the almost sure convergence of the estimator $\widehat{q}_n(x,y)$.

\noindent
\textbf{Proof of Theorem \ref{Q:PDMP:CVps}.} In light of $(\ref{eq:hpq})$, if $p(x)>0$, one may write $q(x,y) = h(x,y)/p(x)$. In addition, $\widehat{q}_n(x,y)$ is defined by the ratio $\widehat{h}_n(x,y)/\widehat{p}_n(x)$ (see $(\ref{eq:exprhn})$ for the expression of $\widehat{h}_n(x,y)$ and $(\ref{eq:exprpn})$ for the one of $\widehat{p}_n(x)$), where $\widehat{h}_n(x,y)$ (respectively $\widehat{p}_n(x)$) estimates $h(x,y)$ (resp. $p(x)$). In such a case, the result is a corollary of Propositions \ref{prop:pchapeaun} and \ref{prop:hchapeaun}.


\section{Central limit theorem}\label{s:clt}

	In this section, we establish a central limit theorem for $\widehat{q}_n(x,y)$ for any $(x,y)\in E^2$ (see Theorem \ref{Q:PDMP:TCL}). First, we derive the rate of convergence of the recursive estimator $\widehat{h}_n(x,y)$ of $h(x,y)$ under some conditions on the parameter $\beta$.
	
	\begin{prop1}\label{prop:vitessehn}
	Let us choose $v_1$ and $w_1$ such that $\max(v_1,w_1)\delta<\text{\normalfont dist}(x,\partial E)$. If
	$$2(1-\alpha d)<4\beta<\min\left(\frac{1}{2d} \, ,\, \alpha - \frac{1}{2d}\right),$$
	then, when $n$ goes to infinity,
	$$n^{(1-\alpha d)/2} \left(\widehat{h}_n(x,y) - h(x,y)\right) \tops 0 .$$
	\end{prop1}
	
	\begin{proof}
	We take back the definitions $(\ref{eq:defYn})$ of $Y_n$ and $(\ref{eq:defVn})$ of $V_n$ used in the proof of Proposition \ref{prop:hchapeaun}. Recall that we have $V_n^2=Y_n$. Together with $(\ref{eq:exprhn})$ and $(\ref{eq:exprAn})$, we have
	\begin{equation*}
	n^{(1-\alpha d)/2}\left(\widehat{h}_n(x,y) - h(x,y)\right) = n^{(1-\alpha d)/2} V_n + n^{(1-\alpha d)/2} \left(\esp\left[\widehat{h}_n(x,y)\right] - h(x,y)\right).
	\end{equation*}
	By $(\ref{eq:biais:maj})$ and $w_j=w_1 j^{-\beta}$,
	$$n^{(1-\alpha d)/2} \left(\esp\left[\widehat{h}_n(x,y)\right] - h(x,y)\right) = \mathcal{O}\left(n^{1/2-\beta-\alpha d/2}\right).$$
	As a consequence, the bias term 
	$$n^{(1-\alpha d)/2} \left(\esp\left[\widehat{h}_n(x,y)\right] - h(x,y)\right)$$
	tends to $0$ when $2\beta>1-\alpha d$. Now, we only have to prove that $W_n=n^{1-\alpha d} Y_n$ almost surely tends to $0$. By virtue of Van Ryzin's lemma and a calculus similar to the one implemented in the proof of Proposition \ref{prop:hchapeaun}, the sequence $(W_n)$ almost surely tends to a random variable $W$ for $\beta$ such that $4\beta d<\alpha d-1/2$. As aforementioned in the proof of Proposition \ref{prop:hchapeaun},
	$$\esp[W_n]\,=\,n^{1-\alpha d}\esp[Y_n]\,=\,n^{1-\alpha d}\var\left(\widehat{h}_n(x,y)\right).$$
	Thus, in light of Corollary \ref{cor:varHn}, $(W_n)$ converges to $0$ in $\mathbf{L}^1$ when $4\beta<\alpha$. Consequently, $W=0~a.s.$ This shows the expected result.
	\end{proof}

	Under an additional assumption presented in the sequel, we give a central limit theorem for $\widehat{p}_n(x)$.
	
	\begin{hyp}\label{hyps:clt}For any $\alpha$ such that $1/(2+d)<\alpha<1/d$, we assume that, when $n$ goes to infinity,
	\begin{eqnarray*}
	\frac{1}{n^{(1+\alpha d)/2}}\left|\sum_{j=1}^{n} \left(r(Z_j^-,x)-p(x)\right)\right| &\tops& 0 .
	\end{eqnarray*}
	\end{hyp}

	\noindent
	This kind of condition is satisfied for iterative models with some Lipschitz mixing properties (see \cite[page 234]{MD} for instance). We have the following central limit theorem for $\widehat{p}_n(x)$.

	\begin{prop1}\label{tcl:pn}
	Let us choose $v_1$ such that $v_1\delta<\text{\normalfont dist}(x,\partial E)$. If $1/(2+d)<\alpha<1/d$ and $p(x)>0$, then, 
	\begin{equation*}\label{clt:eq:pn}
	n^{(1-\alpha d)/2} \left(\widehat{p}_n(x)-p(x)\right) \stackrel{\mathcal{D}}{\longrightarrow} \mathcal{N}\left(0,\frac{\tau^2p(x)}{1+\alpha d}\right) ,
	\end{equation*}
	when $n$ goes to infinity, where $\tau^2=\int_{\R^d}K^2(y)\ud y$.
	\end{prop1}
	
	\begin{proof}
	By $(\ref{eq:decompoh})$, we have
	\begin{align*}
	n^{(1-\alpha d)/2} &\big(\widehat{p}_n(x) - p(x)\big)\\
	&= \frac{1}{n^{(1+\alpha d)/2}v_1^d} K\left(\frac{Z_1^- - x}{v_1}\right) + \frac{M_n}{n^{(1+\alpha d)/2}}+ 	n^{(1-\alpha d)/2} R_n^{(1)} + 	n^{(1-\alpha d)/2} R_n^{(2)},
	\end{align*}
	where $M_n$, $R_n^{(1)}$ and $R_n^{(2)}$ are defined by $(\ref{eq:decompo:Mn})$, $(\ref{eq:decompo:Rn1})$ and $(\ref{eq:decompo:Rn2})$. It is obvious that the first term almost surely tends to $0$ since $K$ is a bounded function. In addition, by $(\ref{eq:decompo:majorRn1})$ together with $v_j=v_1 j^{-\alpha}$,
	$$ n^{(1-\alpha d)/2} \left|R_n^{(1)}\right| = \mathcal{O}\left(n^{\frac{1-\alpha d - 2\alpha}{2}}\right)~a.s.$$
	As a consequence, $n^{(1-\alpha d)/2} R_n^{(1)}$ almost surely tends to $0$ when $1<\alpha (d+2)$. The term $n^{(1-\alpha d)/2} R_n^{(2)}$ almost surely converges to $0$ under Assumption \ref{hyps:clt} with $(\ref{eq:exprh})$. By $(\ref{eq2:crochetMn})$ and the almost sure ergodic theorem, we have
	$$\frac{(1+\alpha d)\langle M\rangle_n}{n^{1+\alpha d}} \tops \tau^2 \int_{\overline{E}} r(\xi,x)\pi(\ud\xi)\,=\, \tau^2 p(x).$$
	As a consequence, by virtue of the central limit theorem for martingales (see for instance Corollary 2.1.10 of \cite{MD}), we have
	\begin{equation*}
	\frac{M_n}{n^{(1+\alpha d)/2}} \stackrel{\mathcal{D}}{\longrightarrow} \mathcal{N}\left(0 , \frac{\tau^2 p(x)}{1+\alpha d}\right) ,
	\end{equation*}
	when $n$ goes to infinity, if Lindeberg's condition is satisfied. Now, we only have to verify this technical condition in order to end the proof. By $(\ref{eq:decompo:Mn})$ together with $v_j=v_1 j^{-\alpha}$, we have
	$$\left|M_j - M_{j-1}\right|\,\leq\,\frac{\|K\|_{\infty}}{v_{j+1}^d} + \|r\|_{\infty} \,=\,\mathcal{O}\left(j^{\alpha d}\right)~a.s.$$ 
	Thus, there exists a constant $C$ such that we have the following inclusions, for $1\leq j\leq n$,
	$$\left\{ |M_j-M_{j-1}| \geq \varepsilon n^{(1+\alpha d)/2}\right\} \,\subset\, \left\{ C j^{\alpha d} \geq \varepsilon n^{(1+\alpha d)/2}\right\} \,\subset\,\left\{ C n^{\alpha d} \geq \varepsilon n^{(1+\alpha d)/2}\right\} ,$$
	for any $\varepsilon>0$. Consequently, as $n^{(\alpha d-1)/2}$ tends to $0$ when $\alpha d<1$, we have
	$$\mathbf{1}_{\left\{|M_j-M_{j-1}|\geq \varepsilon n^{(1+\alpha d)/2}\right\}}   =  0~a.s.$$
	for $n$ large enough. Finally, when $n$ goes to infinity,
	$$\frac{1}{n^{1+\alpha d}} \sum_{j=1}^n\esp\left[ (M_j-M_{j-1})^2 \mathbf{1}_{\left\{|M_j-M_{j-1}|\geq \varepsilon n^{(1+\alpha d)/2}\right\}} \Bigg|\mathcal{F}_j\right] \tops 0 .$$
	This shows Lindeberg's condition and thus the expected result.
	\end{proof}

	Now, one may give the proof of the central limit theorem for $\widehat{q}_n(x,y)$, presented in Theorem \ref{Q:PDMP:TCL}, under some assumptions on the bandwidth parameters $\alpha$ and $\beta$.
	
%
\noindent
\textbf{Proof of Theorem \ref{Q:PDMP:TCL}.} We have
\begin{eqnarray*}
n^{(1-\alpha d)/2} &&\hspace{-0.9cm}\left(\widehat{q}_n(x,y)-q(x,y)\right)\\
&=& n^{(1-\alpha d)/2} \left( \frac{\widehat{h}_n(x,y)}{\widehat{p}_n(x)} - \frac{h(x,y)}{p(x)}\right)\\
&=& \frac{n^{(1-\alpha d)/2} \big(\widehat{h}_n(x,y) - h(x,y)\big)}{\widehat{p}_n(x)} + \frac{q(x,y) n^{(1-\alpha d)/2} \big(\widehat{p}_n(x)-p(x)\big)}{\widehat{p}_n(x)} .
\end{eqnarray*}
In light of Propositions \ref{tcl:pn} and \ref{prop:pchapeaun} and Slutsky's theorem, we have
\begin{equation*}
\frac{q(x,y) n^{(1-\alpha d)/2} \left(\widehat{p}_n(x)-p(x)\right)}{\widehat{p}_n(x)}  \stackrel{\mathcal{D}}{\longrightarrow} \mathcal{N}\left(0,\frac{q(x,y)^2\tau^2}{p(x)(1+\alpha d)}\right) ,
\end{equation*}
when $n$ goes to infinity. By virtue of Proposition \ref{prop:vitessehn}, we obtain the expected result.

\vspace{0.5cm}
	
\noindent
\textbf{Acknowledgments.} The author would like to thank the Associate Editor, two anonymous Reviewers and his PhD advisors, François Dufour and Anne Gégout-Petit, for their suggestions and helpful comments.